\documentclass[12pt,draftcls,onecolumn]{IEEEtran}
\usepackage[utf8]{inputenc} 
\usepackage[T1]{fontenc}    
\usepackage{dsfont}
\usepackage[maxfloats=100]{morefloats}[2015/07/22]
\usepackage{multirow}
\usepackage{wrapfig}
\usepackage[T1]{fontenc}
\usepackage{lmodern}
\usepackage{booktabs}       
\usepackage{amsfonts}       
\usepackage{nicefrac}       
\usepackage{microtype}      
\usepackage{enumerate}
\usepackage{comment}
\usepackage{lipsum}
\usepackage{mathtools}
\usepackage{cuted}
\usepackage{float}
\usepackage[dvipsnames,table,xcdraw]{xcolor}
\usepackage{bbm}
\usepackage{varioref}
\usepackage{hyperref}
                                
\usepackage{amssymb} 
\usepackage{rotating}
\usepackage{graphicx}
\usepackage{subcaption}
\usepackage[normalem]{ulem}

\usepackage{amsthm}
\DeclareUnicodeCharacter{00A0}{~}
\usepackage{algorithm}
\usepackage{algpseudocode}
\algnewcommand{\Initialize}[1]{%
	\State \textbf{Initialization:}
	\Statex {\raggedright #1}
}
\newtheorem{assumption}{Assumption}
\newtheorem{theorem}{Theorem}
\newtheorem{lemma}{Lemma}
\newtheorem{proposition}{Proposition}
\newtheorem{definition}{Definition}

\theoremstyle{plain}
\newtheorem{remark}{Remark}

\newcommand{\fy}[1]{{\color{black}#1}}

\newcommand{\mytfrac}{\tfrac}

\title{\LARGE \bf An Incremental Gradient Method for Optimization Problems with Variational Inequality Constraints}

\author{Harshal~D.~Kaushik$^{1}$, \thanks{This first author is with Bayer and can be contacted at {\tt\small (e-mail: harshaldkaushik@gmail.com).} The second and third authors are with the Department of Industrial and Systems Engineering, Rutgers University, {\tt\small (e-mail: ss4303@scarletmail.rutgers.edu, farzad.yousefian@rutgers.edu)}, respectively.}
\and Sepideh Samadi$^{2}$, \and Farzad~Yousefian$^{3}$  
\thanks{Yousefian acknowledges the support of NSF CAREER grant ECCS$-1944500$, ONR grant N00014-22-1-2757, and DOE grant DE-SC0023303.}
\thanks{A very preliminary version of this work appeared in the 2021 American  Control Conference (ACC), IEEE, New Orleans, LA, USA, 2021.\cite{KaushikYousefianACC21}
}}

\begin{document}
\sloppy
\maketitle
\thispagestyle{empty}
\pagestyle{plain}

\maketitle

\begin{abstract}
We consider \fy{minimizing a sum of agent-specific nondifferentiable merely convex functions over the solution set of a variational inequality (VI) problem in that} each agent is associated with a local monotone mapping. \fy{This problem finds an application in computation of the best equilibrium in nonlinear complementarity problems arising in transportation networks}. We develop an iteratively regularized incremental gradient method where at each iteration, agents communicate over a cycle graph to update their solution iterates using their local information about the objective and the mapping. The proposed method is single-timescale in the sense that it does not involve any excessive hard-to-project computation per iteration.  We derive non-asymptotic agent-wise convergence rates for the suboptimality of the global objective function and infeasibility of the VI constraints measured by a suitably defined dual gap function. The proposed method appears to be the first fully iterative scheme equipped with iteration complexity that can address distributed optimization problems with VI constraints \fy{over cycle graphs}. \fy{Preliminary numerical experiments for a transportation network problem and a support vector machine model are presented}. 
\end{abstract}


\section{Introduction}\label{sec:introduction}
\setcounter{section}{1}

\fy{Consider} a system with $m$ agents where the $i^{\text{th}}$ agent is associated with a component function $f_i:\mathbb{R}^n \to \mathbb{R}$ and a mapping $F_i: \mathbb{R}^n \to \mathbb{R}^{n}$. \fy{We consider} the following distributed constrained optimization problem.
\begin{alignat}{2}\label{prob:initial_problem}\tag{P}
&{\text{minimize }} \quad \textstyle\sum\nolimits_{i=1}^m f_i(x) \\
 &\text{subject to} \quad x \in \text{SOL}\left(X,\textstyle\sum\nolimits_{i=1}^mF_i \right), \nonumber
\end{alignat}
\noindent  where $X \subseteq \mathbb{R}^n$ is a set and $\text{SOL}\left(X,\sum\nolimits_{i=1}^mF_i \right)$ denotes the solution set of the variational inequality $\text{VI}\left(X,\sum\nolimits_{i=1}^mF_i \right)$ defined as follows: $x \in X$ solves $ \text{VI}\left(X,\sum\nolimits_{i=1}^mF_i \right)$ if we have $(y-x)^T\sum\nolimits_{i=1}^mF_i(x)\geq 0$ for all $y \in X$. Problem \eqref{prob:initial_problem} represents a distributed optimization framework in the sense that the information about $f_i$ and $F_i$ is locally known to the $i^{\text{th}}$ agent, while the set $X$ is globally known. 
Model \eqref{prob:initial_problem} captures the canonical formulation of distributed optimization  
\begin{alignat}{2}\label{prob:canonical_DO_problem} 
&{\text{minimize }} \quad \textstyle\sum\nolimits_{i=1}^m f_i(x) \\
 &\text{subject to} \quad x \in X, \nonumber
\end{alignat}
that has been extensively studied. Indeed, by choosing $F_i(x):= 0_n$ for all $i$, model \eqref{prob:initial_problem} is equivalent to model \eqref{prob:canonical_DO_problem}. 

{\noindent {\bf Motivating example:}} Nonlinear complementarity problems (NCP) have been employed to formulate diverse applications in engineering and economics. The celebrated Wardrop's principle of equilibrium in traffic networks and also, the Walras's law of competitive equilibrium in economics are among important examples that can be represented using NCP (cf.~\cite{FerrisPang97}). Formally, NCP is defined as follows. Given a mapping $F:\mathbb{R}_+^n\to \mathbb{R}^n$, $x \in \mathbb{R}^n$ solves $\text{NCP}(F)$ if $0\leq x \perp F(x) \geq  0$, where $\perp$ denotes the perpendicularity operator between two vectors. It is known that $\text{NCP}(F)$ can be cast as $\text{VI}( \mathbb{R}^n_{+},F)$ (see Proposition 1.1.3 in~\cite{FacchineiPang2003}). In many applications where $F$ is merely monotone, $\text{NCP}(F)$ may admit multiple equilibria. In such cases, one may consider finding the best equilibrium with respect to a global metric $f:\mathbb{R}^n \to\mathbb{R}$. For example in traffic networks, the total travel time of the network users can be considered as the objective $f$. In fact, the problem of computing the best equilibrium of an NCP is important to be addressed particularly in the design of transportation networks where there is a need to estimate the efficiency of the equilibrium \cite{NisanBook2007,JohariThesis}. In this regime, the goal is to minimize $f(x)$ where $x$ solves $\text{NCP}(F)$. 
Consider a stochastic NCP \fy{associated with the mapping $\mathbb{E}[F(\bullet,\xi(\omega))]$} 
where $\xi:\Omega \to \mathbb{R}^d$ is a random variable associated with the probability space $(\Omega, \mathcal{F},\mathbb{P})$ and $F:\mathbb{R}_+^n \times \Omega \to \mathbb{R}^n$ is a stochastic single-valued mapping. Let $\mathcal{S}_i$ denote a local index set of independent and identically distributed samples from the random variable $\xi$. Employing a sample average approximation scheme, one can consider a distributed NCP given by $x \geq 0$, $\textstyle\sum\nolimits_{i=1}^m \textstyle\sum\nolimits_{\ell \in \mathcal{S}_i} F(x,\xi_\ell) \geq 0$, $x^T \left( \textstyle\sum\nolimits_{i=1}^m \textstyle\sum\nolimits_{\ell \in \mathcal{S}_i} F(x,\xi_\ell)\right)=0$. 
Let $f:\mathbb{R}_+^n \times \Omega \to \mathbb{R}^n$ denote a stochastic objective function that measures the performance of a given equilibrium at a realization of $\xi$. Then, the problem of distributed computation of the best equilibrium of the preceding NCP is formulated as
\begin{alignat}{2}\label{prob:dist_best_NCP} 
&{\text{minimize }} \quad \textstyle\sum\nolimits_{i=1}^m \textstyle\sum\nolimits_{\ell \in \mathcal{S}_i} f(x,\xi_\ell) \\
 &\text{subject to} \quad x \in \mbox{SOL}\left(\mathbb{R}^n_+,\textstyle\sum\nolimits_{i=1}^m \textstyle\sum\nolimits_{\ell \in \mathcal{S}_i} F(\bullet,\xi_\ell)\right). \notag
\end{alignat}
\noindent \fy{Model} \eqref{prob:initial_problem} captures problem \eqref{prob:dist_best_NCP} by defining $X\triangleq \mathbb{R}^n_+$, $f_i(x) \triangleq\sum\nolimits_{\ell \in \mathcal{S}_i} f(x,\xi_\ell)$, and $F_i(x) \triangleq\sum\nolimits_{\ell \in \mathcal{S}_i} F(x,\xi_\ell) $. 

	
{\noindent {{\bf  Scope and literature review:}} }
In addressing the proposed formulation \eqref{prob:initial_problem}, our focus in this paper lies in the development of an incremental gradient (IG) method. IG methods are among popular avenues for addressing the classical model \eqref{prob:canonical_DO_problem}~\cite{nedichThesis, wang2015,GurbuzbalanOzdaglarParrilo2017,GurbuzbalanOzdaglarParrilo2019}. In these schemes, utilizing the additive structure of the problem, the algorithm cycles through the data blocks and updates the local estimates of the optimal solution in a sequential manner \cite{BertsekasNLPBook2016,BertsekasIG11}. In addressing {the} constrained problems {with easy-to-project constraint sets}, the projected incremental gradient (P-IG) method and its subgradient variant were developed~\cite{NedBert2001}. The P-IG scheme is described as follows. Given an initial point $x_{0,1} \in X$ where $X \subseteq \mathbb{R}^n$ denotes the constraint set, for each $k \geq 0$, consider the update rules given by
 \begin{align*}
 	&x_{k,i+1}: = \mathcal{P}_X\left(x_{k,i}-\gamma_{k}\nabla f_i\left(x_{k,i}\right)\right), \quad {\text{for }i \in [m]},\\
 	&x_{k+1,1}: = x_{k,m+1}, 
 \end{align*} 
 where ${x_{k,i} \in \mathbb{R}^n}$ denotes agent $i$'s local copy of the decision variables at iteration $k$, $\mathcal{P}$ denotes the Euclidean projection operator defined as $\mathcal{P}_X(z)\triangleq \text{argmin}_{x\in X}\|x-z\|_2$, and $\gamma_k>0$ denotes the \fy{stepsize, and $[m]\triangleq \{1,\ldots,m\}$}. 
Recently, under strong convexity and twice continuous \fy{differentiability}, and also, boundedness of the generated iterates, the standard IG method was proved to converge with the rate $\mathcal{O}(1/k)$ in the unconstrained case \cite{GurbuzbalanOzdaglarParrilo2019}. This {is an} improvement to the {previously} known rate of $\mathcal{O}(1/\sqrt{k})$ for the merely convex case. Accelerated variants of IG schemes with provable convergence speeds were {also developed}, including  the incremental aggregated gradient method (IAG) \cite{BlattHeroGauchman2007,GurbuzbalanOzdaglarParrilo2017}, SAG \cite{RouxSchmidtBach2012}, and SAGA \cite{DefazioBachJulien2014}. 
\fy{Most of the past research efforts on the design and analysis of IG methods for constrained} optimization problems have focused on addressing easy-to-project sets or sets with {linear} functional inequalities. This has been done through employing duality theory, projection, or penalty methods (see \cite{ChangNedichScaglione2014,AybatHamedani2016,ScutariSun2019, NedichTatarenko2020}).  Also, a celebrated variant of the dual based schemes is the alternating direction method of multipliers (ADMM) (e.g., see \cite{MakhdoumiOzdaglar2017}). Other related papers that have utilized duality theory in distributed constrained regimes include \cite{Bertsekas2015,AybatHamedani2016, HamedaniAybat2019}.  

\noindent {\bf Research gap:} Despite the extensive work in the area of constrained optimization, no provably convergent \fy{single time-scale method exists} in the literature that can be employed to solve distributed optimization problems with VI constraints. \fy{This is mainly because, unlike in standard constrained optimization, the Lagrangian duality theory does not appear to lend itself to be directly employed for addressing VI constraints. The classical scheme for addressing optimization problems with VI constraints is a sequential regularization (SR) framework~\cite[Ch. 12]{FacchineiPang2003} where the regularized problem $\mbox{VI}(X,F+\eta_k\nabla f)$ is solved at every iteration $k$ of the scheme, while $\eta_k$ is updated in the outer iteration and reduced to zero. A key drawback of the SR scheme is that its iteration complexity is unknown and the scheme often convergences very slowly in practice. Moreover, the asymptotic convergence of this scheme is only established when $f$ is strongly convex and smooth.}

 \noindent {\bf Contributions:} \fy{Our main contributions are as follows:}
 
 \noindent  \textit{(i) Complexity guarantees for addressing model \eqref{prob:initial_problem}}: \fy{We develop} an \fy{IG} method equipped with agent-specific iteration complexity guarantees for solving distributed optimization problems with VI constraints of the form \eqref{prob:initial_problem}. To this end, employing a regularization-based relaxation technique, we propose a projected averaging iteratively regularized incremental gradient method (pair-IG) presented by Algorithm \ref{alg:IR-IG_avg}. In Theorem \ref{thm:rates}, under merely convexity of the global objective function and merely monotonicity of the global mapping, we derive new non-asymptotic suboptimality and infeasibility convergence rates for each agent's generated iterates. This implies a total iteration complexity of $\mathcal{O}\left((C_f+C_F)^4\epsilon^{-4}\right)$ for obtaining an $\epsilon$-approximate solution where $C_f$ and $C_F$ denote the bounds on the global objective function's subgradients and the global mapping over {a} compact convex set $X$, respectively. Iterative regularization (IR) has been recently employed as a constraint-relaxation technique in a class of bilevel optimization problems \cite{FarzadPushPull2020,FarzadOMS19} and also in regimes where the duality theory may not be directly applied \cite{FarzadMathProg17,KaushikYousefianSIOPT2021}. Of these, in our recent work \cite{KaushikYousefianSIOPT2021} we employed the IR technique to derive a provably convergent method for solving problem \eqref{prob:initial_problem} in a centralized framework, where the information of the objective function is globally known by the agents. Unlike in \cite{KaushikYousefianSIOPT2021}, here we assume {that} the agents have {access} only {to} local information about both the objective function and the mapping. {\fy{Notably,} this lack of centralized access to information introduces a challenge in both the design and the complexity analysis of the new algorithmic framework in addressing the distributed model \eqref{prob:initial_problem}.} 
 
  \noindent  \textit{(ii) Distributed averaging scheme}: In pair-IG, we employ a distributed averaging scheme where agents can choose their initial averaged iterate arbitrarily and independent from each other. This relaxation in the proposed IG method appears to be novel, even for the classical IG schemes {in addressing \eqref{prob:canonical_DO_problem}}. 
  
  \noindent  \textit{(iii) Rate analysis in the solution space}: Motivated by the recent developments of iterative methods for MPECs \cite{CSY2021MPEC}, it is important to characterize the speed of the proposed scheme in the solution space. To this end, under strong convexity of the global objective function, in Theorem \ref{thm:part2} we derive agent-specific rate statements that compare the generated sequence of each agent with the so-called Tikhonov trajectory.
  


\textbf{Notation:} 
 Throughout, a vector $x \in \mathbb{R}^n$ is assumed to be a column vector and $x^T$ denotes its transpose. \fy{We} use $\|x\|$ to denote the $\ell_2$-norm. For a convex function $f:\mathbb{R}^n\rightarrow\mathbb{R}$ with the domain $\text{dom}(f)$ and  any $x\in\text{dom}(f)$, a vector $ \tilde{\nabla}f(x) \in\mathbb{R}^n$ is called a subgradient of $f$ at $x$ if $f(x) +\tilde{\nabla}f(x)^Tf(y-x)\leq f(y)$ for  all $y\in \text{dom}(f)$. We let $\partial f(x)$ denote the subdifferential set of  function $f$ at $x$.  The Euclidean  projection of vector $x$ onto set $X$ is denoted by $\mathcal{P}_X(x)$. We let $\mathbb{R}^n_{+}$ and $\mathbb{R}^n_{++}$ denote the set $\left\{x\in \mathbb{R}^n\ |\ x \geq 0\right \}$ and $\left\{x\in \mathbb{R}^n\ |\ x >0\right \}$, respectively. Given a set $S \subseteq \mathbb{R}^n$, we let $\mbox{int}(S)$ denote the interior of $S$. 

\section{Algorithm outline}\label{sec:algo_outline}

{In this section we present the main assumptions on problem \eqref{prob:initial_problem}, the outline of the proposed algorithm, and a few preliminary results that will be applied later in the rate analysis. Throughout the paper, we let $f(x)\triangleq \sum_{i=1}^m f_i(x)$ and $F(x)\triangleq \sum_{i=1}^m F_i(x)$ denote the global objective and \fy{global mapping}, respectively. 
\begin{assumption}[Properties of problem \eqref{prob:initial_problem}]\label{assum:initial_prob} 
	\noindent (a) {Function $f_i:\mathbb{R}^n\rightarrow\mathbb{R}$ is real-valued and merely convex (possibly nondifferentiable) on its domain for all {$i \in [m]$}.}

	\noindent (b) Mapping $F_i:\mathbb{R}^n\rightarrow\mathbb{R}^n$ is real-valued, continuous, and merely monotone on its domain for all  {$i \in [m]$}. 
	
	\noindent (c) The set $X \subseteq \mbox{int}\left(\mbox{dom}(f)\cap \mbox{dom}(F)\right) $ is nonempty, convex, and compact.
\end{assumption}
\begin{remark}\label{rem:nonemptiness_SOL}\em
Assumption \ref{assum:initial_prob} \fy{immediately implies the following}. From~\cite[Theorem 2.3.5 and Corollary 2.2.5]{FacchineiPang2003}, $\mbox{SOL}(X,F)$ is nonempty, convex, and compact. For all $i$, the nonemptiness of the subdifferential set $\partial f_i(x)$ for any $x \in \mbox{int}(\mbox{dom}(f_i))$ is implied from Theorem 3.14 in \cite{Beck2017}. Also, Theorem 3.16 in \cite{Beck2017} implies that $f_i$ has bounded subgradients over the compact set $X$. Further, mapping $F_i$ is bounded over the set $X$. 
\end{remark}
In view of compactness of the set $X$ and continuity of $f$, throughout the paper we let positive scalars $M_X< \infty$ and $M_f < \infty$ be defined as $M_X \triangleq \sup_{x \in X} \|x\|$ and $M_f \triangleq \sup_{x \in X} |f(x)|$, respectively.  We also let $f^*\in \mathbb{R}$ denote the optimal objective value of problem \eqref{prob:initial_problem}. In view of Remark \ref{rem:nonemptiness_SOL}, throughout we let scalars $C_F>0$ and $C_f>0$ be defined such that for all {$i\in [m]$} and for all $x \in X$ we have $ \|F_i(x)\| \leq \mytfrac{C_F}{m}$, and $\left\|\tilde{\nabla}{ f_{i}\left( x\right)}\right\|\leq \mytfrac{C_f}{m}$ for all $\tilde{\nabla}{ f_{i}\left( x\right)} \in \partial f_i(x)$. In the following, we comment on the Lipschitz continuity of the local and global objective functions.
\begin{remark}\label{rem:lipschitz_parameters}
\em Under Assumption \ref{assum:initial_prob} and from Theorem 3.61 in \cite{Beck2017}, function $f_i$ is Lipschitz continuous with the parameter $\mytfrac{C_f}{m}$ over the set $X$, i.e., for all {$i \in [m]$} we have $|f_i(x)-f_i(y)|\leq \mytfrac{C_f}{m}\|x-y\|$ for all $x,y \in X$. We also have $\|\tilde \nabla f(x) \| \leq C_f$ for all $x \in X$ and all $\tilde \nabla f(x) \in \partial f(x)$. This implies that $|f(x)-f(y)|\leq C_f\|x-y\|$ for all $x,y \in X$. 
\end{remark}}
We now present an overview of the proposed method given by Algorithm \ref{alg:IR-IG_avg}.  We use vector $x_{k,i}$ to denote the local copy of the global decision vector maintained by agent $i$ at iteration $k$.  At each iteration, agents update their iterates in a cyclic manner.  Each agent {$i\in[m]$} uses only its local information including the subgradient of the function $f_i$ and mapping $F_i$ and evaluates the regularized mapping $F_i + \eta_k \tilde \nabla f_i$ at $x_{k,i}$.  Here, $\gamma_{k}$ and $\eta_k$ denote the stepsize and the regularization parameter at iteration $k$, respectively.  Importantly, through employing an iterative regularization technique, we let both of these parameters be updated iteratively at suitable prescribed rates (cf.  Theorem \ref{thm:rates}).  Each agent computes and returns a weighted averaging iterative denoted by $\bar x_{k,i}$ where the weights are characterized in terms of the stepsize $\gamma_k$ and an arbitrary scalar $r \in [0,1)$. Notably,  this averaging technique is carried out in a distributed fashion in the sense that agents do not require to start from the same initialized averaging iterate. 
\begin{algorithm}[t]
	\caption{projected averaging iteratively regularized Incremental subGradient (pair-IG)}\label{alg:IR-IG_avg}
\begin{algorithmic}[0] 
	 \State \textbf{input}: Agent $1$ arbitrarily chooses an initial vector $ x_{0,1} \in X$. Agent $i$ arbitrarily chooses $\bar x_{0,i} \in X$, for all $i \in [m]$.  Let $S_{0} := \gamma_{0}^r$ with an arbitrary $0\leq r <1$.
		\For {$k = 0,1, \dots, {N-1}$}
		\State Update $S_{k+1}:=S_{k}+\gamma_{k+1}^r $
		\For {$i = 1, \dots, {m}$}
		\State \begin{align}
 		&x_{k,{i+1}} :=  \mathcal{P}_X \left({x_{k,i}}-\gamma_k\left(F_i\left(x_{k,i}\right)+ {\eta_{k}} \tilde{\nabla} f_{i}\left( x_{k,i}\right)\right)\right)\label{eqn:alg_step_1}\\
		&{\bar{x}_{k+1,i}:=\left(\tfrac{S_k}{S_{k+1}}\right)\bar{x}_{k,i}  + \left(\tfrac{\gamma_{k+1}^r}{S_{k+1}}\right) x_{k,i+1} } \label{eqn:alg_step_2}
		\end{align}
	 \EndFor
		\State Set $x_{k+1,1}:= x_{k,{m+1}}$
	 \EndFor
 	\State \textbf{return}: $\bar{x}_{N,i}$ for all $i\in [m]$ 
	\end{algorithmic}
\end{algorithm}
Next we show that {for any $i  \in [m]$, $\bar{x}_{N,i}$} is indeed a well-defined weighted average of $\bar x_{0,i}$ and {the iterates} $x_{k-1,i+1}$ for {$k\in [N]$}.
{\begin{lemma}\label{lem:averaging_seq}
	Consider the sequence {$\{\bar x_{k,i}\}$} generated by agent {$i\in [m]$} in Algorithm \ref{alg:IR-IG_avg}. For {$k\in \{0,\ldots,N\}$}, let us define the weights $\lambda_{k,N} \triangleq \tfrac{\gamma_k^r}{\sum_{j = 0}^{N} \gamma_j^r }$. Then for all {$i \in [m] $} we have 
	\begin{align*}
	\bar{x}_{N,i} = \lambda_{0,N}\bar  x_{0,i} + \textstyle\sum_{k=1}^{N}\lambda_{k,N}x_{k-1,i+1}.\end{align*}
	 Further, for a convex set $X$ we have $\bar{x}_{N,i} \in X$. 
\end{lemma}}
\begin{proof}
	We use induction on $N\geq 0$ to show the equation.  For $N = 0$, from $\lambda_{0,0} = 1$ we have $\bar{x}_{0,i} = \lambda_{0,0}\bar x_{0,i}$. Now,  assume that the {equation} holds for some $N\geq 0$. {This implies}
	\begin{align}\label{eqn:avg_1}
		\bar{x}_{N,i} &= \lambda_{0,N} \bar x_{0,i} +  \textstyle\sum_{k=1}^{N}\lambda_{k,N}x_{k-1,i+1}\nonumber\\ &=  (\gamma_0^r \bar   x_{0,i} +  \textstyle\sum_{k=1}^{N}\gamma_k^r x_{k-1,i+1})/({\textstyle\sum_{j = 0}^{N} \gamma_j^r }).
	\end{align}
	Using equation \eqref{eqn:avg_1}, we now show that the hypothesis statement holds for any $N+1$.  From equation \eqref{eqn:alg_step_2} we have
		$\bar{x}_{N+1,i}=\left(\mytfrac{S_N}{S_{N+1}}\right)\bar{x}_{N,i}  + \left(\mytfrac{\gamma_{N+1}^r}{S_{N+1}}\right) x_{N, i+1}$.  Note that from {equation \eqref{eqn:alg_step_2}} in Algorithm \ref{alg:IR-IG_avg} we have $S_k=\sum_{t=0}^{k}\gamma_{t}^r$ for all $k\geq 0$.  From this and using equation \eqref{eqn:avg_1} we obtain
	\begin{align*}
	\bar{x}_{N+1,i}&=\left(\mytfrac{\sum_{t=0}^{N}\gamma_{t}^r}{\sum_{t=0}^{N+1}\gamma_{t}^r}\right)\bar{x}_{N,i}  + \left(\mytfrac{\gamma_{N+1}^r}{\sum_{t=0}^{N+1}\gamma_{t}^r}\right) x_{N,i+1} \\
	&= \mytfrac{\gamma_0^r \bar  x_{0,i} +  \sum_{k=1}^{N}\gamma_k^r x_{k-1,i+1} + \gamma_{N+1}^r x_{N,i+1}}{\sum_{t=0}^{N+1}\gamma_{t}^r}\\&= \mytfrac{\gamma_0^r \bar  x_{0,i} +  \sum_{k=1}^{N+1}\gamma_k^r x_{k-1,i+1}}{\sum_{t=0}^{N+1}\gamma_{t}^r}.	
	\end{align*}	
From the definition of $\lambda_{k,N}$ we conclude that the hypothesis holds for $N+1$ and thus, the result holds for all $N\geq 0$.  To show the second part, note that from the {initialization in Algorithm \ref{alg:IR-IG_avg}} and the projection in equation \eqref{eqn:alg_step_1}, we have $\bar x_{0,i}, x_{k-1,i+1}\in X$ for all $i$ and $k\geq 1$. From the first part, $\bar{x}_{N,i}$ is a convex combination of $\bar x_{0,i}, x_{0,i+1}, \ldots,x_{N-1,i+1}$. Therefore, \fy{$\bar{x}_{N,i} \in X$ using the convexity of the set $X$.}
\end{proof}
For the ease of presentation throughout the analysis, we define a sequence $\{x_k\}$ as follows.
\begin{definition}\label{def:convenient_notation}
Consider Algorithm \ref{alg:IR-IG_avg}. Let $\{x_k\}$ be {defined as  
$x_k \triangleq x_{k-1, m+1}=x_{k,1},$ for all $k \geq 1,$ with $x_0\triangleq x_{0,1}.$}
\end{definition}
In the following result, we characterize the distance between the local variable of {any} arbitrary agent with that of the first and the last agent at any given iteration. 
\begin{lemma}\label{lem:neighbor_agents}
Consider Algorithm \ref{alg:IR-IG_avg}. Let Assumption \ref{assum:initial_prob} hold. Then the following inequalities hold for all {$i\in[m]$ and $k\geq 0$}

\noindent (a) $\| x_k - x_{k,i} \|\leq \mytfrac{(i-1)\gamma_{k}\left(C_F+\eta_{k}C_f\right)}{m}$.

\noindent (b) $	\left\| x_{k,i+1} -  x_{k+1} \right\| \leq  \mytfrac{(m-i){\gamma_{k}}\left({C_F} + {\eta_{k}} C_f\right)}{m}$. 
\end{lemma}
\begin{proof}
\noindent (a) Let $k\geq 0$ be an arbitrary integer. We use induction on $i$ to show this result. From Definition \ref{def:convenient_notation},  for $i=1$ and $k\geq 0$ we have $\|x_k - x_{k,1}\| = 0$, implying that the result holds for $i=1$. Now suppose the hypothesis statement holds for some {$i\in [m]$}. We have
	\begin{align*}
		&\| x_k - x_{k,i+1} \| \\&=  \left\| \mathcal{P}_X \left(x_{k}\right) - \mathcal{P}_X \left( x_{k,i} - \gamma_{k}\left(F_i(x_{k,i}) + \eta_{k}\tilde{\nabla} f_i\left(x_{k,i}\right)\right)  \right)  \right\|\nonumber\\   &\leq  \left\| x_{k} - x_{k,i} \right\| +  \gamma_{k}\left\| F_i(x_{k,i}) + \eta_{k}\tilde{\nabla} f_i\left(x_{k,i}\right)  \right\|\nonumber\\ 
		& \leq  \left\| x_{k} - x_{k,i} \right\| +  \mytfrac{\gamma_{k}\left(C_F + \eta_k C_f\right)}{m} \leq  \mytfrac{i\gamma_{k}\left(C_F + \eta_k C_f\right)}{m},	
	\end{align*}
where the first inequality is obtained from the nonexpansivity property of the projection. Therefore the hypothesis statement holds for any $i$ and the proof of part (a) is completed. 

\noindent (b) To show this result, we use downward induction on {$i\in[m]$}. Note that the relation trivially holds for the base case $i=m$. Suppose it holds for some {$i \in \{2,\ldots,m\}$}. We show that it holds for $i-1$ as well.  From Definition \ref{def:convenient_notation} we have
\begin{align*}
	\left\| x_{k,i} -  x_{k+1} \right\| &= \| x_{k,i} - x_{k,i+1} +x_{k,i+1} -  x_{k,m+1} \| \\&\leq \| x_{k,i} - x_{k,i+1} \| + \|x_{k,i+1} -  x_{k,m+1}\|.
\end{align*}
{From equation \eqref{eqn:alg_step_1}, the hypothesis statement, and the nonexpansivity property of the projection, we obtain}
\begin{align*}
	&\left\| x_{k,i} -  x_{k+1} \right\|\\ 
	&\leq \left\| \mathcal{P}_X( x_{k,i}) - \mathcal{P}_X \left({x_{k,i}}-{\gamma_{k}}\left(F_i\left(x_{k,i}\right)+ {\eta_{k}} \tilde{\nabla} f_{i}\left( x_{k,i}\right)\right)\right)\right\| \\
	&+ \mytfrac{(m-i){\gamma_{k}}\left({C_F} + {\eta_{k}} C_f\right)}{m}
	\leq \mytfrac{(m-i+1){\gamma_{k}}\left({C_F} + {\eta_{k}} C_f\right)}{m}.
\end{align*}
This completes the proof of part (b).
\end{proof}
We note that the generated agent-wise iterates $\bar x_{k,i}$ in Algorithm \ref{alg:IR-IG_avg}, as the scheme proceeds,  may not be solutions to $\mbox{VI}(X,F)$ and so, they may not necessarily be feasible to problem \eqref{prob:initial_problem}.  To quantify the infeasibility of these iterates, we employ a dual gap function (cf. Chapter 1 in \cite{FacchineiPang2003}) defined as follows. 

\begin{definition}[The dual gap function {\cite{FacchineiPang2003}}]\label{def:gap}\normalfont
Consider a closed convex set $X\subseteq \mathbb{R}^n$ and  a continuous mapping $F:X\rightarrow\mathbb{R}^n$. The dual gap function at $x\in X$ is defined as $\mbox{GAP}(x) \triangleq \sup_{y\in X} F(y)^T(x-y)$.
\end{definition}
Note that under Assumption \ref{assum:initial_prob}, the dual gap function is well-defined. This is because for any $x\in X$,  we have $\mbox{GAP}(x)\geq 0$. Further, it is known that when the mapping $F$ is continuous and monotone, $x \in \text{SOL}(X,F)$ if and only if $\mathrm{GAP}(x)=0$ (cf.~\cite{Nem11}). We conclude this section by presenting the following result that will be utilized in the rate analysis. 
\begin{lemma}[\fy{\cite[Lemma 2.14]{KaushikYousefianSIOPT2021}}]\label{lem:harmonic_series_bound} Let $\beta \in [0,  1)$, $\Gamma\geq 1$, and $K$ be an integer. Then for all \fy{$K\geq \left(\sqrt[1-\beta]{2}-1\right)\Gamma$}, we have $
\mytfrac{{(K+\Gamma)}^{1-\beta}}{2(1-\beta)} \leq \textstyle\sum_{k = 0}^{K}(k+\Gamma)^{-\beta}\leq \mytfrac{{(K+\Gamma)}^{1-\beta}}{1-\beta}.$
\end{lemma}
\section{Rate and complexity analysis}\label{sec:convergence_analysis}
{In this section we present the convergence and rate analysis of the proposed method under Assumption \ref{assum:initial_prob}.  After obtaining a preliminary inequality in Lemma \ref{lem:pseudo_bound_sequence} in terms of the sequence generated by the last agent, in Lemma \ref{lem:consensus_error_bound} we derive inequalities that relate the global objective and the dual gap function at the iterate of other agents with those of the last agent. Utilizing these results, in Proposition \ref{prop:error_bound_aIR-IG} we obtain agent-specific bounds on the objective function value and the dual gap function. Consequently, in Theorem \ref{thm:rates} we derive convergence rate statements under suitably chosen sequences for the stepsize and the regularization parameter. }
{\begin{lemma}\label{lem:pseudo_bound_sequence}
	Consider Algorithm \ref{alg:IR-IG_avg}. Let Assumption \ref{rem:nonemptiness_SOL} hold. Let $\{\gamma_{k}\}$ and $\{\eta_{{k}}\}$ be nonincreasing and strictly positive sequences.  For any arbitrary $y\in X$, for all $k\geq 0$ we have
	\begin{align}\label{eq_36}
			&{2\gamma_{k}^r\left(\eta_{k}\left(f(x_{k}) -f(y) \right)   +F(y)^T\left(x_{k}-y\right)\right)}\leq \gamma_{k}^{r-1} \left\| {x_{k}}-y\right\|^2 \nonumber\\
			& -\gamma_{k}^{r-1}\left\| x_{k+1} -  y \right\|^2     +\gamma_{k}^{r+1}\left(C_F +\eta_{{k}}C_f\right)^2.
	\end{align}
\end{lemma}}
\begin{proof}
	Let $y\in X$ be an arbitrary vector and $k\geq 0$ be fixed. From the update rule \eqref{eqn:alg_step_1}, for {$i \in  [m]$} we have
	\begin{align*}
	&\left\| x_{k,{i+1}} -  y \right\|^2 \\
	&=\left\| \mathcal{P}_X\left({x_{k,i}}-\gamma_k\left( F_i\left(x_{k,i}\right)    + {\eta_{k}} \tilde{\nabla} f_{i}\left( x_{k,i}\right)\right)\right)- \mathcal{P}_X (y) \right\|^2.
	\end{align*}
	{Employing the nonexpansivity of the projection we have}
	\begin{align*}
	&\left\| x_{k,{i+1}} -  y \right\|^2  \leq \left\| {x_{k,i}}-\gamma_k\left(F_i\left(x_{k,i}\right)+ {\eta_{k}} \tilde{\nabla} f_{i}\left( x_{k,i}\right)\right)-  y \right\|^2\\
&	=  \left\| {x_{k,i}}-y\right\|^2  + \gamma_k^2\left\|F_i\left(x_{k,i}\right)+ {\eta_{k}} \tilde{\nabla} f_{i}\left( x_{k,i}\right) \right\|^2  \\&-  2\gamma_{k}\left(F_i\left(x_{k,i}\right)+ {\eta_{k}} \tilde{\nabla} f_{i}\left( x_{k,i}\right)\right)^T\left(x_{k,{i}} -  y \right).
	\end{align*}
	From the triangle inequality and recalling the  bounds on $\tilde\nabla f_i(x)$ and $F_i(x)$, we obtain
	\begin{align}\label{eqn:lemma2_1}
		&\left\| x_{k,{i+1}} -  y \right\|^2 \leq  \left\| {x_{k,i}}-y\right\|^2 + \gamma_{k}^2\left(\mytfrac{C_F + \eta_{{k}}C_f}{m}\right)^2\nonumber\\ 
		&+ 2\gamma_{k}\left(F_i\left(x_{k,i}\right)+ {\eta_{k}} \tilde{\nabla} f_{i}\left( x_{k,i}\right)\right)^T\left(y-x_{k,{i}}  \right).
	\end{align}
	{The last term in the preceding relation is  bounded as follows}
	\begin{align*}
&2\gamma_{k}\left(F_i\left(x_{k,i}\right)+ {\eta_{k}} \tilde{\nabla} f_{i}\left( x_{k,i}\right)\right)^T\left(y-x_{k,{i}}  \right)  \\
&= 2\gamma_{k}F_i\left(x_{k,i}\right)^T\left(y-x_{k,{i}}  \right)+2\gamma_{k} {\eta_{k}} \tilde{\nabla} f_{i}\left( x_{k,i}\right)^T\left(y-x_{k,{i}}  \right)\\
& \leq 2\gamma_{k}F_i\left(y\right)^T\left(y-x_{k,{i}}  \right)+2\gamma_{k} {\eta_{k}}\left( f_{i}\left( y\right)- f_{i}\left(x_{k,i}\right)\right),
	\end{align*}
where the last inequality is implied from the monotonicity of  $F_i$ and convexity of $f_i$. {Combining with equation \eqref{eqn:lemma2_1} we have}
	\begin{align*}
	&\left\| x_{k,{i+1}} -  y \right\|^2   \leq \left\| {x_{k,i}}-y\right\|^2 + \gamma_{k}^2\left(\mytfrac{C_F + \eta_{{k}}C_f}{m}\right)^2\\
	&+ 2\gamma_{k}F_i\left(y\right)^T\left(y-x_{k,{i}}  \right)+2\gamma_{k} {\eta_{k}}\left( f_{i}\left( y\right)- f_{i}\left(x_{k,i}\right)\right).
	\end{align*}
Adding and subtracting $2\gamma_{k}F_i\left(y\right)^Tx_{k}+2\gamma_{k} {\eta_{k}} f_{i}\left( x_{k}\right)$ we {get}
	\begin{align*}
	&\left\| x_{k,{i+1}} -  y \right\|^2   
	\leq  \left\| {x_{k,i}}-y\right\|^2 + \gamma_{k}^2\left(\mytfrac{C_F + \eta_{{k}}C_f}{m}\right)^2\\
	&+ 2\gamma_{k}F_i\left(y\right)^T\left(y-x_{k}  \right)+2\gamma_{k} {\eta_{k}}\left( f_{i}\left( y\right)- f_{i}\left( x_{k}\right)  \right)\\
	& + 2\gamma_{k}\left( \left|F_i\left(y\right)^T\left(x_{k}-x_{k,i}  \right)\right| + {\eta_{k}} \left| f_{i}\left( x_k\right)- f_{i}\left( x_{k,i}\right)\right|\right).
	\end{align*}	
	Using the Cauchy-Schwarz inequality and Remark \ref{rem:lipschitz_parameters} we obtain
	\begin{align*}
	&\left\| x_{k,{i+1}} -  y \right\|^2 \leq  \left\| {x_{k,i}}-y\right\|^2 + \gamma_{k}^2\left(\mytfrac{C_F + \eta_{{k}}C_f}{m}\right)^2\\
	&+ 2\gamma_{k}F_i\left(y\right)^T\left(y-x_{k}  \right)+2\gamma_{k} {\eta_{k}}\left( f_{i}\left( y\right)- f_{i}\left( x_{k}\right)  \right) \\
	&+ 2\gamma_{k}\left(\mytfrac{C_F}{m}\left\|x_{k}-x_{k,i}  \right\|+\mytfrac{ {\eta_{k}}C_f}{m}\left\|x_k-  x_{k,i}\right\|\right).
	\end{align*}
	{Summing over $i\in[m]$} and considering Definition  \ref{def:convenient_notation} we have
	\begin{align*}
	&\left\| x_{k+1} -  y \right\|^2  \leq  \left\| {x_{k}}-y\right\|^2 {+ \mytfrac{\gamma_{k}^2\left(C_F + \eta_{{k}}C_f\right)^2}{m}}\\
	&+ 2\gamma_{k}F\left(y\right)^T\left(y-x_{k}  \right)+2\gamma_{k} {\eta_{k}}\left( f\left( y\right)- f\left( x_{k}\right)  \right)\\
	&+ \mytfrac{2\gamma_{k}\left(C_F+{\eta_{k}}C_f\right)}{m}\textstyle\sum_{i = 1}^m \left\|x_{k}-x_{k,i}  \right\|{.}
	\end{align*}
From Lemma \ref{lem:neighbor_agents} we obtain
	\begin{align*}
		&\left\| x_{k+1} -y \right\|^2 \leq\left\| {x_{k}}-y\right\|^2 + {\mytfrac{\gamma_{k}^2\left(C_F + \eta_{{k}}C_f\right)^2}{m}}  \\ &+2\gamma_kF(y)^T\left(y-x_{k}\right) + 2\gamma_{k}\eta_{k}\left(f(y) - f(x_{k})\right)  \\
		&+ \mytfrac{2 \gamma_{k}\left(C_F +\eta_{k}C_f\right)}{m}\textstyle\sum_{i=1}^m \mytfrac{(i-1)\gamma_{k}\left(C_F + \eta_k C_f\right)}{m}   \\
	    &=  \left\| {x_{k}}-y\right\|^2 + \gamma_{k}^2\left(C_F+\eta_{{k}}C_f\right)^2 +2\gamma_kF(y)^T\left(y-x_{k}\right)\\
	    &+  2\gamma_{k}\eta_{k}\left(f(y) - f(x_{k})\right).
		\end{align*}
Multiplying the both sides by $\gamma_{k}^{r-1}$ {we can obtain} the result.
\end{proof}
In the next result we provide inequalities that relate the objective function and the dual gap function at the generated averaged iterate of the last agent with that of any other agent, respectively. This result will be utilized in Proposition \ref{prop:error_bound_aIR-IG}.
\begin{lemma}\label{lem:consensus_error_bound}
Consider problem \eqref{prob:initial_problem} and the sequences $\{\bar{x}_{N,i}\}$ generated in Algorithm \ref{alg:IR-IG_avg} for {$i \in [m]$} for some $N\geq 1$. Let Assumption \ref{assum:initial_prob} hold {and let $\{\gamma_k\}$ and $\{\eta_k\}$} be strictly positive {and} nonincreasing sequences. Then for any {$i \in [m]$} we have 
\begin{subequations}
	\begin{align}
		&f(\bar x_{N,i}) - f(\bar x_{N,m}) \leq C_f\lambda_{0,N}\left\|\bar   {x}_{0,i} -\bar   {x}_{0,m}\right\|\nonumber\\
		&+ \mytfrac{(m-i)C_f\left({C_F} + \eta_0 C_f\right)}{m}\textstyle\sum_{{k=0}}^{N} \lambda_{k,N}\gamma_{k},\label{eqn:consensus_error_bound_f} \\
		&\mathrm{GAP}\left(\bar x_{N,i}\right) - \mathrm{GAP}(\bar x_{N,m})  \leq {C_F}\lambda_{0,N}\left\| \bar  {x}_{0,i} -\bar  {x}_{0,m}\right\|\nonumber\\
		&+ \mytfrac{(m-i){C_F}\left({C_F} + \eta_0 C_f\right)}{m}\textstyle\sum_{{k=0}}^{N} \lambda_{k,N}\gamma_{k},\label{eqn:consensus_error_bound_GAP}
	\end{align}
\end{subequations}
where $\lambda_{k,N} \triangleq \tfrac{\gamma_k^r}{\sum_{j = 0}^{N} \gamma_j^r }$ for {$k \in \{0,\ldots,N\}$}.
\end{lemma}
\begin{proof} {Note that the results are trivial when $m=1$. Throughout, we assume that $m\geq 2$. From the Lipschitz continuity of function $f$ from Remark \ref{rem:lipschitz_parameters} and invoking Lemma \ref{lem:averaging_seq},  we can write the following for all {$i \in [m]$}.
{\begin{align}\label{eqn:consensus_error_begin}
		&f(\bar x_{N,i}) - f(\bar x_{N,m}) \leq C_f\lambda_{0,N}\left\| \bar {x}_{0,i} -\bar  {x}_{0,m}\right\| \nonumber\\
	&+C_f \textstyle\sum_{k=1}^{N} \lambda_{k,N} \left\| x_{k-1,i+1} -  x_{k-1,m+1} \right\|.
\end{align}}
Next, using Lemma~\ref{lem:neighbor_agents}(b) for any $k\geq 1$ and {$i\in[m]$} we have
\begin{align}\label{eqn:induction_bound_consensus}
	\left\| x_{k-1,i+1} -  x_{k-1,m+1} \right\| \leq  \mytfrac{(m-i){\gamma_{k-1}}\left({C_F} + {\eta_{k-1}} C_f\right)}{m}. 
	\end{align}
From \eqref{eqn:consensus_error_begin},  \eqref{eqn:induction_bound_consensus}, and the nonincreasing sequence $\{\eta_k\}$, we have
\begin{align*}
	f(\bar x_{N,i}) - f(\bar x_{N,m})  
	& \leq  \mytfrac{(m-i)C_f\left({C_F} + {\eta_0} C_f\right)}{m}\sum_{k=1}^{N} \lambda_{k,N} {\gamma_{k-1}} \\& + C_f\lambda_{0,N}\left\|  \bar {x}_{0,i} -  \bar {x}_{0,m}\right\|.
\end{align*}
{Since $\{\gamma_k\}$ is nonincreasing  and $0\leq r <1$, we obtain}
\begin{align*}
&\textstyle\sum_{k=1}^{N} \lambda_{k,N} {\gamma_{k-1}}  \leq \mytfrac{1}{\textstyle\sum_{j=0}^N \gamma_j^r}\textstyle\sum_{k=1}^{N}  {\gamma_{k-1}^{r+1}}  \\& 
\leq \mytfrac{1}{\textstyle\sum_{j=0}^N \gamma_j^r}\textstyle\sum_{k=0}^{N}  {\gamma_{k}^{r+1}} =\textstyle\sum_{k=0}^{N} \lambda_{k,N} {\gamma_{k}}.
\end{align*}
From the last two relations we obtain equation \eqref{eqn:consensus_error_bound_f}. Next we show \eqref{eqn:consensus_error_bound_GAP}. From Definition \ref{def:gap} we have
\begin{align*}
	&\mathrm{GAP}(\bar x_{N,i}) = \sup_{y \in X}F(y)^T\left(\bar x_{N,i} - y\right) \\
	& = \sup_{y \in X} F(y)^T\left(\bar x_{N,i} + \bar x_{N,m} -\bar x_{N,m} - y\right)	\\
	& \leq\sup_{y \in X} F(y)^T\left(\bar x_{N,i} - \bar x_{N,m}\right) + \sup_{y \in X}  F(y)^T\left( \bar x_{N,m} - y\right)\\
	& \leq {C_F} \left\|\bar x_{N,i} - \bar x_{N,m}\right\| + \mathrm{GAP}\left( \bar x_{N,m} \right),
\end{align*}
Rearranging the terms we obtain $\mathrm{GAP}(\bar x_{N,i}) -  \mathrm{GAP}\left( \bar x_{N,m} \right) \leq {C_F} \left\|\bar x_{N,i} - \bar x_{N,m}\right\|$. The rest of the proof can be done in a similar fashion to the proof of \eqref{eqn:consensus_error_bound_f}.}
\end{proof}
{Next we construct agent-wise error bounds in terms of the {objective function value} and the  dual gap function at the averaged iterates generated in Algorithm \ref{alg:IR-IG_avg}.
\begin{proposition}[Agent-wise error bounds]\label{prop:error_bound_aIR-IG}
	Consider problem \eqref{prob:initial_problem} and  the averaged sequence $\{\bar{x}_{k,i}\}$ generated by agent $i$ in Algorithm \ref{alg:IR-IG_avg} for {$i \in [m]$}. Let Assumption \ref{assum:initial_prob} hold and $\{\gamma_{k}\}$ and $\{\eta_{k}\}$ be nonincreasing and strictly positive sequences. {Then we have    for  $i \in [m]$,  $N\geq 1$, and $r \in [0,1)$:}
\begingroup
\allowdisplaybreaks
\begin{align*}
& (a)\ f(\bar{x}_{N,i}) -f^*\leq\left({\textstyle\sum_{k = 0}^{N} \gamma_{k}^r}\right)^{-1}\left(\mytfrac{2M_X^2\gamma_{N}^{r-1}}{\eta_{N}}  \right. \\ &  \left.  +  {\mytfrac{\left({C_F}+\eta_{0}C_f\right)^2}{2}} \textstyle\sum_{k = 0}^{N}\mytfrac{\gamma_{k}^{r+1}}{\eta_k} + \mytfrac{(m-i)C_f\left(C_F + \eta_0 C_f\right)}{m}\textstyle\sum_{k=0}^{N}\gamma_{k}^{r+1}\right.\\
&\left.+{\gamma_{0}^rf(\bar  x_{0,m}) - \gamma_{0}^rf(x_{0,1})}+{C_f\gamma_{{0}}^r\left\| \bar {x}_{0,i} -\bar  {x}_{0,m}\right\|} \right).\\
&(b)\ \mathrm{GAP}(\bar{x}_{N,i})  \leq\left(\textstyle\sum_{k = 0}^{N} \gamma_{k}^r\right)^{-1}\left(2M_X^2  \gamma_{N}^{r-1} +2M_f\textstyle\sum_{k = 0}^{N}\gamma_{k}^r\eta_{k}\right.\\
&\left. + { \mytfrac{\left(C_F+C_f\eta_{{0}}\right)^2}{2}\textstyle\sum_{k=0}^{N}\gamma_{k}^{r+1}} + \mytfrac{(m-i)C_F\left(C_F + \eta_0 C_f\right)}{m}\textstyle\sum_{k=0}^{N}\gamma_{k}^{r+1}\right.\\
&\left.+\gamma_0^rC_F\|\bar x_{0,m}-x_{0,1} \|+{C_F\gamma_{{0}}^r\left\| \bar {x}_{0,i} -\bar  {x}_{0,m}\right\|} \right).
\end{align*}
\endgroup
\end{proposition}}
\begin{proof}
		{(a) Let $x^* \in X$ denote an arbitrary optimal solution to problem \eqref{prob:initial_problem}. From feasibility of $x^*$ we have  $F(x^*)\left(x_k-x^*\right)\geq 0$. Substituting $y$ by $x^*$ in relation \eqref{eq_36} and using the preceding relation we have
		\begin{align*}
			2\gamma_{k}^r\eta_{k}\left(f(x_{k}) -f^* \right)  &\leq  \gamma_{k}^{r-1} \left\| {x_{k}}-x^*\right\|^2 -\gamma_{k}^{r-1}\left\| x_{k+1} -  x^* \right\|^2 \\&+  {\gamma_{k}^{r+1}\left(C_F +\eta_{{k}}C_f\right)^2} .
		\end{align*}
	Dividing both sides by $2\eta_k$ we have
		\begin{align}\label{eq_38}
			 \gamma_{k}^r\left(f(x_{k}) -f^* \right)& \nonumber \leq    \mytfrac{\gamma_{k}^{r-1}}{2\eta_{{k}}}\left( \left\| {x_{k}}-x^*\right\|^2  - \left\| x_{k+1} -  x^* \right\|^2\right) \\& +{ \mytfrac{\gamma_{k}^{r+1}}{2\eta_k}\left({C_F}+\eta_{k}C_f\right)^2}.
		\end{align}
		Adding and subtracting {the term} $\mytfrac{\gamma_{k-1}^{r-1}}{2\eta_{k-1}}\|x_k-x^*\|^2$ we have
		\begin{align}\label{eqn:prop_1_1}
			&\gamma_{k}^r\left(f(x_{k}) -f^* \right) \leq  \mytfrac{\gamma_{k-1}^{r-1}}{2\eta_{k-1}}\|x_k-x^*\|^2\nonumber -\mytfrac{\gamma_{k}^{r-1}}{2\eta_{{k}}}\left\| x_{k+1} -  x^* \right\|^2 \\ 
			&\hspace{-0.25cm} + \underbrace{\left(\mytfrac{\gamma_{k}^{r-1}}{2\eta_{k}} - \mytfrac{\gamma_{k-1}^{r-1}}{2\eta_{k-1}}\right)}_{\text{term 1}}\|x_k-x^*\|^2 + { \mytfrac{\gamma_{k}^{r+1}}{2\eta_k}\left({C_F}+\eta_{k}C_f\right)^2}.
		\end{align}
		Recalling {the definition of scalar $M_X$} we have
		\begin{align}\label{eq:bound_X}
			\|x_k-x^*\|^2 \leq 2\|x_k\|^2 + 2\|x^*\|^2  \leq 4M_X^2. 
		\end{align}
		Taking into account $r < 1$,  the nonincreasing property of
		the sequences $\{\gamma_k\}$ and $\{\eta_k\}$, we have: term 1 $\geq 0 $. Using \eqref{eq:bound_X} and taking summation from \eqref{eqn:prop_1_1} over {$k \in [N]$},  we obtain
		\begin{align}\label{eq_39}
			&\sum_{k = 1}^{N} \gamma_{k}^r\left(f(x_{k}) -f^* \right)\leq  \mytfrac{\gamma_{0}^{r-1}}{2\eta_{0}}\|x_1-x^*\|^2-\mytfrac{\gamma_{N}^{r-1}}{2\eta_{N}}\|x_{N+1}-x^*\|^2\nonumber\\ 
			&+\left(\mytfrac{\gamma_{N}^{r-1}}{2\eta_{N}}-\mytfrac{\gamma_{0}^{r-1}}{2\eta_{0}}\right)4M_X^2+{\mytfrac{\left({C_F}+\eta_{0}C_f\right)^2}{2} \sum_{k = 1}^{N}\mytfrac{\gamma_{k}^{r+1}}{\eta_k}}.
		\end{align}
		Rewriting equation \eqref{eq_38} for $k = 0$ and then, adding and subtracting $f(\bar x_{0,m})$, we have
		\begin{align*}
		&\gamma_{0}^r\left(f(\bar x_{0,m})  -f^* + f(x_{0}) - f(\bar x_{0,m}) \right)\leq    \mytfrac{\gamma_{0}^{r-1} \left\| {x_{0}}-x^*\right\|^2}{2\eta_{{0}}}  \\
		&-\mytfrac{\gamma_{0}^{r-1}\left\| x_{1} -  x^* \right\|^2}{2\eta_{{0}}}    + {\left({C_F}+\eta_{0}C_f\right)^2 \mytfrac{\gamma_{0}^{r+1}}{2\eta_0}}.
		\end{align*}
		Adding the preceding equation with \eqref{eq_39} we obtain
		\begin{align*}
		&\gamma_{0}^r\left(f(\bar x_{0,m})  -f^* \right)	+\sum_{k = 1}^{N} \gamma_{k}^r\left(f(x_{k}) -f^* \right)\leq  \mytfrac{\gamma_{0}^{r-1} \left\| {x_{0}}-x^*\right\|^2}{2\eta_{{0}}} \\
		&+ 2M_X^2\left(\mytfrac{\gamma_{N}^{r-1}}{\eta_{N}}-\mytfrac{\gamma_{0}^{r-1}}{\eta_{{0}}}\right) -\mytfrac{\gamma_{N}^{r-1}}{2\eta_{N}} \left\| x_{N+1} -  x^* \right\|^2  \\
		& +{\mytfrac{\left({C_F}+\eta_{0}C_f\right)^2}{2} \sum_{k = 0}^{N}\mytfrac{\gamma_{k}^{r+1}}{\eta_k}}+{\gamma_{0}^r\left(f(\bar x_{0,m})  -f(x_0) \right)}.
		\end{align*}
{From  \eqref{eq:bound_X}} and neglecting the nonpositive term we obtain
		\begin{align*}
			  &\gamma_{0}^r\left(f(\bar x_{0,m})  -f^* \right)	+\textstyle\sum_{k = 1}^{N} \gamma_{k}^r\left(f(x_{k}) -f^* \right)\leq  \mytfrac{2M_X^2\gamma_{N}^{r-1}}{\eta_{N}}  \\
			  & + {\mytfrac{\left({C_F}+\eta_{0}C_f\right)^2}{2} \textstyle\sum_{k =0}^{N}\mytfrac{\gamma_{k}^{r+1}}{\eta_k}}+{\gamma_{0}^r\left(f(\bar x_{0,m})  -f(x_0) \right)}.
		\end{align*}
		Next, dividing both sides by  $\sum_{k = 0}^{N} \gamma_{k}^r$ we have
		\begin{align*}
			&\mytfrac{{\gamma_{0}^r f(\bar x_{0,m}) }	+\textstyle\sum_{k = 1}^{N} \gamma_{k}^rf(x_{k}) }{\sum_{k = 0}^{N} \gamma_{k}^r} -f^*
			 \leq\left(\textstyle\sum_{k = 0}^{N} \gamma_{k}^r\right)^{-1}\left(\mytfrac{2M_X^2\gamma_{N}^{r-1}}{\eta_{N}}\right.\\
			&\left.+ \mytfrac{\left({C_F}+\eta_{0}C_f\right)^2}{2} \textstyle\sum_{k = 0}^{N}\mytfrac{\gamma_{k}^{r+1}}{\eta_k}  +{\gamma_{0}^r\left(f(\bar x_{0,m})  -f(x_0) \right)}\right).
		\end{align*}
Taking into account the convexity of  $f$ we have
\begin{align*}
f\left(\mytfrac{\gamma_{0}^r\bar x_{0,m} + \sum_{k = 1}^{N} \gamma_{k}^rx_{k-1,m+1}}{\sum_{k = 0}^{N} \gamma_{k}^r}\right)  \le\mytfrac{\gamma_{0}^rf(\bar x_{0,m}) + \sum_{k = 1}^{N} \gamma_{k}^rf(x_{k-1,m+1})}{\sum_{k = 0}^{N} \gamma_{k}^r}.
\end{align*} 
Invoking Lemma \ref{lem:averaging_seq}, from the preceding two relations we obtain
	\begin{align*}
		&f(\bar{x}_{N,m}) -f^* \leq   \left( \textstyle\sum_{k = 0}^{N} \gamma_{k}^r\right)^{-1}\left(\mytfrac{2M_X^2\gamma_{N}^{r-1}}{\eta_{N}}\right.\\
		&\left.+  {\mytfrac{\left({C_F}+\eta_{0}C_f\right)^2}{2}} \textstyle\sum_{k = 0}^{N}\mytfrac{\gamma_{k}^{r+1}}{\eta_k}   +{\gamma_{0}^rf( \bar x_{0,m}) - \gamma_{0}^rf(x_{0,1})}\right).
	\end{align*}
Adding equation \eqref{eqn:consensus_error_bound_f} with the preceding inequality we obtain the desired result. 
		
\noindent(b) From equation \eqref{eq_36}, for an arbitrary $y\in X$ we have
		\begin{align*}
		&2\gamma_k^rF(y)^T\left(x_{k}-y\right)   \leq   \gamma_{k}^{r-1} \left(\left\| {x_{k}}-y\right\|^2-\left\| x_{k+1} - y \right\|^2\right)\\
		& + 2\gamma_{k}^r\eta_{k}\left(f\left(y\right) - f(x_{k})  \right)+ { \gamma_{k}^{r+1}\left(C_F +\eta_{{k}}C_f\right)^2}.
		\end{align*}
From the triangle inequality and definition of $M_f$ we have $|f(y) - f(x_k)|\leq 2M_f$. We obtain:
		\begin{align}\label{eq39}
			&	2\gamma_k^rF(y)^T\left(x_{k}-y\right)   \leq  \gamma_{k}^{r-1} \left(\left\| {x_{k}}-y\right\|^2  -\left\| x_{k+1} - y \right\|^2\right)  \nonumber\\
				& + 4\gamma_{k}^r\eta_{k}M_f + \gamma_{k}^{r+1}\left(C_F +\eta_{{k}}C_f\right)^2.
		\end{align}
Adding and subtracting $\gamma_{k-1}^{r-1}\|x_k-y\|^2$, we have:
		\begin{align}\label{eqn:prop1_partb_ineq1}
			 &2\gamma_k^rF(y)^T\left(x_{k}-y\right)   \leq  \gamma_{k-1}^{r-1} \left\| {x_{k}}-y\right\|^2-\gamma_{k}^{r-1}\left\| x_{k+1} -  y \right\|^2 + \nonumber \\
			 &4\gamma_{k}^r\eta_{k}M_f+ \underbrace{\left(\gamma_{k}^{r-1} -\gamma_{k-1}^{r-1} \right) \left\| {x_{k}}-y\right\|^2}_{\text{term 2}} +  {\gamma_{k}^{r+1}\left(C_F +\eta_{{k}}C_f\right)^2}  .
		\end{align}
		Using the nonincreasing property of $\{\gamma_{k}\}$ and recalling $0\leq r<1$, we have $ \gamma_{k}^{r-1} -\gamma_{k-1}^{r-1} \geq 0 $. Thus, we can write: term 2 $\leq \left( \gamma_{k}^{r-1} -\gamma_{k-1}^{r-1} \right) 4M_X^2 $. Taking summation over {$k\in [N]$ in equation \eqref{eqn:prop1_partb_ineq1}} and dropping a nonpositive term we obtain
		\begin{align}\label{eq40}
			 &2\sum_{k = 1}^{N}\gamma_k^rF(y)^T\left(x_{k}-y\right) \leq  \gamma_{0}^{r-1} \left\| {x_{1}}-y\right\|^2 +4M_f\sum_{k = 1}^{N}\gamma_{k}^r\eta_{k} \nonumber \\   
			 &+ 4M_X^2 \left(\gamma_{N}^{r-1} -\gamma_{0}^{r-1} \right)      + { \left(C_F+\eta_{{0}}C_f\right)^2\sum_{k = 1}^{N}\gamma_{k}^{r+1}}.
\end{align}
Writing equation \eqref{eq39} for $k=0$ and adding and subtracting $2\gamma_0^rF(y)^T\bar x_{0,m}$, we have
		\begin{align*}
			& 2\gamma_0^rF(y)^T\left(\bar x_{0,m}-y + x_0 -\bar  x_{0,m}\right)    \leq  4\gamma_{0}^r\eta_{0}M_f  \\
			 & + \gamma_{0}^{r-1} \left(\left\| {x_{0}}-y\right\|^2   -\left\| x_{1} -  y \right\|^2\right) +  {\gamma_{0}^{r+1}\left(C_F+\eta_{{0}}C_f\right)^2}\nonumber.
		\end{align*}
Adding the preceding relation with equation \eqref{eq40} we have
		\begin{align*}
			&2\gamma_0^rF(y)^T\left(\bar x_{0,m}-y \right)  + 2\sum_{k = 1}^{N}\gamma_k^rF(y)^T\left(x_{k}-y\right) \\
			& \leq  4M_X^2 \left( \gamma_{N}^{r-1}- \gamma_{0}^{r-1} \right)   + \left(C_F+\eta_{{0}}C_f\right)^2\sum_{k=0}^{N}\gamma_{k}^{r+1}\\
			&+ \gamma_{0}^{r-1} \left\| {x_{0}}-y\right\|^2+4M_f\sum_{k = 0}^{N}\gamma_{k}^r\eta_{k}+{\underbrace{2\gamma_0^rF(y)^T\left(\bar x_{0,m}-x_0 \right)}_{\text{term 3}}} .
		\end{align*}
		Using the Cauchy-Schwarz inequality we have: term 3 $\leq 2\gamma_0^rC_F\|x_{0,m}-x_{0,1} \|$. We also have $\|x_0-y\|^2 \leq 4M_X^2$. Dividing the both sides of the preceding inequality by  $2\sum_{k = 0}^{N} \gamma_{k}^r$ and invoking Lemma \ref{lem:averaging_seq}, we have
		 \begin{align*}
		 	&F(y)^T\left({\bar{x}_{N,m}}-y\right) \leq  \left(\sum_{k = 0}^{N}\gamma_k^r\right)^{-1}\hspace{-0.25cm}\left(\tfrac{2M_X^2}{ \gamma_{N}^{1-r}} +2M_f\textstyle\sum_{k = 0}^{N}\gamma_{k}^r\eta_{k} \right.\\
		 	& \left.  + { \mytfrac{\left(C_F+\eta_{{0}}C_f\right)^2}{2}\textstyle\sum_{k=0}^{N}\gamma_{k}^{r+1}}+{\gamma_0^rC_F\|\bar x_{0,m}-x_{0,1} \|}\right).
		 \end{align*}
	Taking the supremum on both sides with respect to $y$ over the set $X$ and recalling Definition \ref{def:gap} we have
	 \begin{align*}
	 	& \mathrm{GAP}\left({\bar{x}_{N,m}}\right) \leq  \left(\sum_{k = 0}^{N}\gamma_k^r\right)^{-1}\left( {\mytfrac{\left(C_F+\eta_{{0}}C_f\right)^2}{2}\sum_{k=0}^{N}\gamma_{k}^{r+1}}\right.\\
		&\left.+2M_f\sum_{k = 0}^{N}\gamma_{k}^r\eta_{k}+2M_X^2  \gamma_{N}^{r-1}   +{\gamma_0^rC_F\|\bar x_{0,m}-x_{0,1} \|}\right).
	 \end{align*}
Adding equation \eqref{eqn:consensus_error_bound_GAP} with the preceding inequality we obtain the desired inequality.}
\end{proof}
{In the following we present the main result of this section. We provide non-asymptotic rate statements for each agent {$i \in [m]$} in terms of suboptimality measured by the global objective function, and infeasibility characterized by the dual gap function. We note that unlike the analysis of the standard incremental gradient schemes in the literature, here we provide these rate results for individual agents {$i \in [m]$}. 
\begin{theorem}[Agent-wise rate statements for Algorithm \ref{alg:IR-IG_avg}]\label{thm:rates}
Consider problem \eqref{prob:initial_problem}.  Let the averaged sequence $\{\bar{x}_{k,i}\}$ be generated by agent {$i \in[m]$} using Algorithm \ref{alg:IR-IG_avg}. Let Assumption \ref{assum:initial_prob} hold. Let the stepsize sequence $\{\gamma_{k}\}$ and the regularization sequence $\{\eta_{{k}}\}$ be updated using $\gamma_{k} := \mytfrac{\gamma_{0}}{\sqrt{k+1}}$ and $\eta_{k} := \mytfrac{\eta_{0}}{(k+1)^b}$, respectively, where $\gamma_{0}, \eta_{0} > 0$ and $0< b <0.5$. Then the following inequalities hold for all {$i\in[m]$}, all $N\geq 2^{\mytfrac{2}{1-r}}-1$, and all $r \in [0,1)${:}
\begingroup
\allowdisplaybreaks
\begin{align}
(a)\  & f(\bar{x}_{N,i}) -f^* \leq\mytfrac{2-r}{(N+1)^{0.5-b}}\left(\mytfrac{2M_X^2}{\eta_0\gamma_0}+  \mytfrac{\gamma_0\left(C_F+\eta_0C_f\right)^2}{\eta_0(1-r+2b)} \right.\nonumber\\
 & \left. + f( \bar x_{0,m}) - f(x_{0,1})  +C_f\left\|\bar  {x}_{0,i} -\bar  {x}_{0,m}\right\| \right.\nonumber\\
 & \left.+ {\mytfrac{{2(m-i)}\gamma_{0}C_f\left(C_F + \eta_0 C_f\right)}{m(1-r)}} \right).\label{eqn:rate_f}\\ 
(b)\   & \mathrm{GAP}(\bar{x}_{N,i}) \leq 
	 \mytfrac{2-r}{(N+1)^{b}} \left(\mytfrac{2M_X^2}{\gamma_0} +\mytfrac{2M_f \eta_0}{1-0.5r-b} \right.\nonumber\\
	 & \left.+{C_F\|\bar x_{0,m}-x_{0,1} \|}  + C_F\left\| \bar {x}_{0,i} - \bar {x}_{0,m}\right\|\right.\notag\\
		&\left.+ \mytfrac{\left(C_F+\eta_0C_f\right)^2\gamma_0}{1-r}+\mytfrac{2{(m-i)}C_F\left(C_F + \eta_0 C_f\right) \gamma_0}{m(1-r)}\right).		
		\label{eqn:rate_gap}
\end{align}
\endgroup
\end{theorem}}
\begin{proof} 
{(a) Consider the inequality in Proposition \ref{prop:error_bound_aIR-IG}(a). Substituting $\gamma_k$ and $\eta_k$ by their update rules we obtain
\begin{align*}
	&f(\bar{x}_{N,i}) -f^* \leq \left(\textstyle\sum_{k=0}^{N}\mytfrac{\gamma_0^r}{(k+1)^{0.5r}}\right)^{-1}\left(\mytfrac{2M_X^2(N+1)^{0.5(1-r)+b}}{\eta_0\gamma_0^{1-r}} \right.\nonumber\\
	&  \hspace{-0.2cm}\left.  + \mytfrac{\left(C_F+\eta_0C_f\right)^2}{2}\sum_{k=0}^{N}\mytfrac{\gamma_0^{1+r}}{\eta_0(k+1)^{0.5(1+r)-b}}+\gamma_{0}^rf(\bar  x_{0,m}) - \gamma_{0}^rf(x_{0,1})\right.\nonumber\\
	&\hspace{-0.2cm}\left.+ {\mytfrac{{(m-i)}C_f\left(C_F + \eta_0 C_f\right)}{m}\sum_{k=0}^{N}\mytfrac{\gamma_{0}^{r+1}}{(k+1)^{0.5(1+r)}}}+{C_f\gamma_{{0}}^r\left\| \bar {x}_{0,i} -\bar  {x}_{0,m}\right\|} \right).
\end{align*}
In the next step, to apply Lemma \ref{lem:harmonic_series_bound} we need to ensure that the conditions in that result are met. From $0\leq r<1$ and $0<b<0.5$, we have $0\leq 0.5r <1$, $0\leq 0.5(1+r)-b <1$, $0\leq 0.5r+b<1$, and $0 \leq 0.5(1+r)<1$.  Further, from $N\geq 2^{\mytfrac{2}{1-r}}-1$, $0<b<0.5$, and $0\leq r<1$ we have that $N\geq \max\left\{2^{1/(1-0.5r)},2^{1/(1-0.5(1+r)+b)}, 2^{1/(1-0.5(1+r))}\right\}-1$. 

{\noindent Therefore, all the necessary conditions of Lemma \ref{lem:harmonic_series_bound} are met.}
 \begin{align*}
 &	{f(\bar{x}_{N,i}) -f^*} \leq \left(\mytfrac{\gamma_0^r(N+1)^{1-0.5r}}{2(1-0.5r)}\right)^{-1} \left(\mytfrac{2M_X^2(N+1)^{0.5(1-r)+b}}{\eta_0\gamma_0^{1-r}} \right.\nonumber\\
 & \left.   +  \mytfrac{\gamma_0^{1+r}\left(C_F+\eta_0C_f\right)^2(N+1)^{1-0.5(1+r)+b}}{2\eta_0(1-0.5(1+r)+b)} \right.\nonumber\\
 & \left. + {\mytfrac{{(m-i)}C_f\left(C_F + \eta_0 C_f\right)\gamma_{0}^{r+1}(N+1)^{1-0.5(1+r)}}{m(1-0.5(1+r))}} \right.\nonumber\\
 & \left.+{\gamma_{0}^rf( \bar x_{0,m}) - \gamma_{0}^rf(x_{0,1})}+{C_f\gamma_{{0}}^r\left\|\bar  {x}_{0,i} - \bar {x}_{0,m}\right\|} \right).
 \end{align*}
From the preceding relation we obtain
  \begin{align*} 
 &	{f(\bar{x}_{N,i}) -f^*} \leq \left({2-r}\right)\left(\mytfrac{2M_X^2}{\eta_0\gamma_0(N+1)^{0.5-b}} \right.\nonumber\\
 & \left.+  \mytfrac{\gamma_0\left(C_F+\eta_0C_f\right)^2}{2\eta_0(1-0.5(1+r)+b)(N+1)^{0.5-b}}    + {\mytfrac{{(m-i)}C_f\left(C_F + \eta_0 C_f\right)\gamma_{0}}{m(1-0.5(1+r))(N+1)^{0.5}}} \right.\nonumber\\
 &\left.+\mytfrac{f( \bar x_{0,m}) - f(x_{0,1})+C_f\left\|\bar  {x}_{0,i} - \bar {x}_{0,m}\right\|}{(N+1)^{1-0.5r}} \right).
 \end{align*}
Factoring out $1/(N+1)^{0.5-b}$ we obtain
  \begin{align*} 
 &	{f(\bar{x}_{N,i}) -f^*} \leq \mytfrac{2-r}{(N+1)^{0.5-b}}\left(\mytfrac{2M_X^2}{\eta_0\gamma_0}+  \mytfrac{\gamma_0\left(C_F+\eta_0C_f\right)^2}{2\eta_0(1-0.5(1+r)+b)} \right.\nonumber\\
 & \left.   + {\mytfrac{{(m-i)}C_f\left(C_F + \eta_0 C_f\right)\gamma_{0}}{m(1-0.5(1+r))(N+1)^{b}}}+\mytfrac{f(\bar  x_{0,m}) - f(x_{0,1})+C_f\left\| \bar {x}_{0,i} - \bar {x}_{0,m}\right\|}{(N+1)^{0.5-0.5r+b}} \right).
 \end{align*}
Note that from $b>0$ and $r<1$ we have $0.5-0.5r+b>0$.  Hence equation \eqref{eqn:rate_f} holds. 

\noindent (b) Consider the inequality in Proposition \ref{prop:error_bound_aIR-IG}(b). We have
		\begin{align*}
			&\mathrm{GAP}(\bar{x}_{N,i}) \leq
			 \left(\textstyle\sum_{k = 0}^{N}\gamma_k^r\right)^{-1} \left( 
			  2M_X^2  \gamma_{N}^{r-1}   +2M_f\textstyle\sum_{k = 0}^{N}\gamma_{k}^r\eta_{k}	\right.\nonumber\\
			  &\left.+{\gamma_0^rC_F\|\bar x_{0,m}-x_{0,1} \|} +{\gamma_{{0}}^rC_F\left\|\bar  {x}_{0,i} - \bar {x}_{0,m}\right\| }\right.\\
			  &\left.  +\left( \mytfrac{\left(C_F+\eta_{0}C_f\right)^2}{2} + \mytfrac{{(m-i)}C_F\left(C_F + \eta_0 C_f\right)}{m}\right)\textstyle\sum_{k=0}^{N}\gamma_{k}^{r+1}\right).
		\end{align*}
Substituting $\{\gamma_k\}$ and $\{\eta_k\}$ by their update rules we obtain
		\begin{align*}
			&\mathrm{GAP}(\bar{x}_{N,i}) \leq
			\left(\sum_{k=0}^{N}\mytfrac{\gamma_0^r}{(k+1)^{0.5r}}\right)^{-1} \left(\mytfrac{2M_X^2(N+1)^{0.5(1-r)}}{\gamma_0^{1-r}}   \right.\\
			&\left.+\sum_{k=0}^{N}\mytfrac{2M_f \eta_0\gamma_0^r}{(k+1)^{0.5r+b}} +{\gamma_0^rC_F\left(\|\bar x_{0,m}-x_{0,1} \| +\left\|\bar  {x}_{0,i} -\bar  {x}_{0,m}\right\| \right)} \right.\\ 
&\left.+\left( \mytfrac{\left(C_F+\eta_{0}C_f\right)^2}{2} + \mytfrac{{(m-i)}C_F\left(C_F + \eta_0 C_f\right)}{m}\right)\sum_{k=0}^{N}\mytfrac{\gamma_0^{r+1}}{(k+1)^{0.5(1+r)}}\right).
		\end{align*}
Utilizing the bounds in Lemma \ref{lem:harmonic_series_bound} we obtain
	\begin{align*}
		&\mathrm{GAP}(\bar{x}_{N,i}) \leq
		\left(\mytfrac{\gamma_0^r(N+1)^{1-0.5r}}{2(1-0.5r)}\right)^{-1} \left(\mytfrac{2M_f \eta_0\gamma_0^r (N+1)^{1-0.5r-b}}{1-0.5r-b}  \right.\\
		&\left.+\mytfrac{2M_X^2(N+1)^{0.5(1-r)}}{\gamma_0^{1-r}} +{\gamma_0^rC_F\left(\|\bar x_{0,m}-x_{0,1} \|  +\left\|\bar  {x}_{0,i} - \bar {x}_{0,m}\right\|\right)}\right.\\
		&\left.+ {\left( \mytfrac{\left(C_F+\eta_{0}C_f\right)^2}{2} + \mytfrac{{(m-i)}C_F\left(C_F + \eta_0 C_f\right)}{m}\right)\mytfrac{\gamma_0^{r+1}(N+1)^{1-0.5(1+r)}}{(1-0.5(1+r))}}\right).
	\end{align*}	
Rearranging the terms we obtain
	\begin{align*} 
		&\mathrm{GAP}(\bar{x}_{N,i})\leq
		\left({2-r}\right)\left(\mytfrac{2M_X^2}{\gamma_0(N+1)^{0.5}}   +\mytfrac{2M_f \eta_0}{(1-0.5r-b)(N+1)^{b}}\right.\\
		&\left.+\mytfrac{{C_F\|\bar x_{0,m}-x_{0,1} \|}  + C_F\left\| \bar {x}_{0,i} -\bar  {x}_{0,m}\right\|}{(N+1)^{1-0.5r}}\right.\\
		&\left.+ {\left( \mytfrac{\left(C_F+\eta_{0}C_f\right)^2}{2} + \mytfrac{{(m-i)}C_F\left(C_F + \eta_0 C_f\right)}{m}\right)\mytfrac{\gamma_0}{(1-0.5(1+r))(N+1)^{0.5}}}\right).
	\end{align*}	
Note that from $b<0.5$ and $0\leq r<1$ we have $1-0.5r \geq b$.  Hence equation \eqref{eqn:rate_gap} holds. }
\end{proof} 
{\begin{remark}[Iteration complexity of Algorithm \ref{alg:IR-IG_avg}] \normalfont
	Consider the rate results presented by relations \eqref{eqn:rate_f} and \eqref{eqn:rate_gap}. Let us choose $r:=0$ and suppose $\gamma_{k} := \mytfrac{(C_F+C_f)^{-1}}{\sqrt{k+1}}$ and $\eta_{k}:=\mytfrac{1}{\sqrt[4]{k+1}}$ for $k\geq 0$. Let $\epsilon>0$ be an arbitrary small scalar such that $f\left(\bar{x}_{N_\epsilon,i}\right) -f^*+\mathrm{GAP}\left(\bar{x}_{N_\epsilon,i}\right)<\epsilon$ for all {$i\in [m]$}. Then,  we obtain the iteration complexity of $N_\epsilon=\mathcal{O}((C_F+C_f)^4\epsilon^{-4})$ for each agent.  Interestingly, this iteration complexity matches the complexity of the proposed method in our earlier work \cite{KaushikYousefianSIOPT2021} for addressing formulation \eqref{prob:initial_problem} in a centralized regime where the information of the objective function $f$ is globally known. Importantly, this indicates that there is no sacrifice in the iteration complexity in addressing the distributed formulation \eqref{prob:initial_problem}. Another {important} observation to make is that the iteration complexity of the proposed distributed method is independent of the number of agents $m$. 
\end{remark}}

\section{Comparisons with Tikhonov trajectory}\label{sec:convergence_analysis_strong_convexity}
{In this section we study the convergence rate properties of the proposed method in the solution space. To this end, we compare the sequences generated from Algorithm \ref{alg:IR-IG_avg} with the Tikhonov trajectory (\fy{see} Definition \ref{def:regularized_problem}). Throughout this section, we make the following additional assumption.
\begin{assumption}  \label{assum:initial_problem_3} Consider problem \eqref{prob:initial_problem}. For all  {$i \in [m]$} let {the} component function $f_i:\mathbb{R}^n\rightarrow {\mathbb{R}}$ be continuously differentiable and $\mu_{f_i}-$strongly convex over the set $X$.
\end{assumption}
Note that under this assumption, the equation \eqref{eqn:alg_step_1} in Algorithm \ref{alg:IR-IG_avg} can be written as
		\begin{align}\label{eqn:alg_2_step_1}
		x_{k,{i+1}} & :=  \mathcal{P}_X \left({x_{k,i}}-\gamma_k\left(F_i\left(x_{k,i}\right)+ {\eta_{k}} {\nabla} f_{i}\left( x_{k,i}\right)\right)\right).
	\end{align}
{Next, we comment on the strong convexity parameter of the global objective function.}
\begin{remark}\label{rem:strong_convex_parameter_f}\normalfont
Under Assumption \ref{assum:initial_problem_3}, we note that any function $f_i$ is strongly convex with a parameter $\mu_{\text{min}}\triangleq \min_{{i \in [m]}}\mu_{f_i}$. This also implies that the global function $f(x)\triangleq \sum_{i = 1}^m f_i(x)$ is $m \mu_{\text{min}}-$ strongly convex. Another implication is that under Assumption \ref{assum:initial_prob}(b), (c), and Assumption \ref{assum:initial_problem_3}, problem \eqref{prob:initial_problem} has a unique optimal solution. Throughout this section, we denote the {unique optimal} solution of \eqref{prob:initial_problem} by $x^*$.
\end{remark}
\subsection{Preliminaries}\label{subsec:preliminararies}
Here we provide some preliminary results that will be used {later}. We first introduce the notion of Tikhonov trajectory.
\begin{definition}[Tikhonov trajectory]\label{def:regularized_problem} \normalfont
Let $\{\eta_k\}\subseteq \mathbb{R}_{++}^n$ be a given sequence. Consider a class of regularized VI problems for $k\geq 0$ given by
 \begin{align}\label{prob:regularized_problem}
 	\text{VI}\left( X, \textstyle\sum_{i = 1}^m \left(F_i + \eta_{k} {\nabla} f_i\right)\right).
 \end{align}
Let $x^*_{\eta_k}$ denote the unique solution to \eqref{prob:regularized_problem}. The sequence  $\{ x^*_{\eta_k}\}$ is called the Tikhonov trajectory associated with problem \eqref{prob:initial_problem}. 
\end{definition}
In the following, we establish the convergence of the Tikhonov trajectory to the unique optimal solution of problem \eqref{prob:initial_problem}. This result will be used later in Theorem \ref{thm:part2}.
\begin{lemma}[Properties of the Tikhonov trajectory]\label{lem:recursive_bound_reg_problem}
	Consider problem \eqref{prob:initial_problem} and Definition \ref{def:regularized_problem}. Let Assumption \ref{assum:initial_prob}(b), (c), and Assumption \ref{assum:initial_problem_3} hold. Then:\\
	\noindent (a) Let $\{\eta_k\}$ be a strictly positive sequence such that ${\text{lim}}_{k\rightarrow \infty}\eta_{{k}}=0$. Then, $\lim_{k\to\infty} x^*_{\eta_{{k}}}$ exists and is equal to $x^*$.
	
	\noindent (b) For any two nonnegative integers $k_1$ and $k_2$, we have $\left\| x^*_{\eta_{{k_2}}} - x^*_{\eta_{{k_1}}}  \right\|\leq \mytfrac{C_f}{m \mu_{\text{min}}}\left|1-\mytfrac{\eta_{{k_2}}}{\eta_{{k_1}}}\right|.$
\end{lemma}   
\begin{proof}
The proof can be done in a similar fashion to the proof of Lemma 4.5 in \cite{KaushikYousefianSIOPT2021}. 
\end{proof}
Next we obtain a recursive bound on an error metric that is characterized by the sequence $\{x_k\}$ in Definition \ref{def:convenient_notation} and the Tikhonov trajectory. This result will be utilized in Theorem \ref{thm:part2}.
\begin{lemma}
	Let  $\{x_k\}$ be given by Definition \ref{def:convenient_notation}. Let Assumption \ref{assum:initial_prob}(b), (c), and Assumption \ref{assum:initial_problem_3} hold. Suppose $\{\gamma_k\}$ and $\{\eta_k\}$ are strictly positive and nonincreasing such that $\gamma_0\eta_0 \mu_{\text{min}}\leq 0.5$. Then for any $k\geq 1$ we have
	\begin{align}\label{eqn:iterative_bound_0}
	&\left\| x_{k+1} - x^*_{\eta_k} \right\|^2 \leq 
 \left(1- \gamma_k\eta_k \mu_{\text{min}}\right)\left\| x_k-x^*_{\eta_{k-1}} \right\|^2\nonumber\\&+\mytfrac{1.5C_f^2}{m^2\gamma_k\eta_k\mu_{\text{min}}^3}\left(1-\mytfrac{\eta_{k}}{\eta_{k-1}}\right)^2 +\gamma_k^2\left(C_F+\eta_kC_f\right)^2.
	\end{align}
\end{lemma}}
\begin{proof}
{From Algorithm \ref{alg:IR-IG_avg}, the nonexpansivity of the projection, and}   $x^*_{\eta_k} \in X$, for any {$i \in [m]$} and $k\geq 1$ we have
{\begin{align*} 
		&\left\| x_{k,i+1} - x^*_{\eta_k} \right\|^2 \leq  \left\| x_{k,i} - x^*_{\eta_k} \right\|^2\\
		&+\gamma_{k}^2\left\|F_i(x_{k,i}) +\eta_k\nabla f_i\left(x_{k,i}\right)\right\|^2 \\
		&-  2 \gamma_{k}\left(F_i(x_{k,i}) + \eta_k\nabla f_i\left(x_{k,i}\right)\right)^T\left( x_{k,i} - x^*_{\eta_k}\right).
	\end{align*}}
From the definition of $C_F$ and $C_f$ we have
	\begin{align*} 
		&\left\| x_{k,i+1} - x^*_{\eta_k} \right\|^2 
		\leq \left\| x_{k,i} - x^*_{\eta_k} \right\|^2+\gamma_k^2
		\left(\mytfrac{C_F+\eta_kC_f}{m}\right)^2+  \\&2 \gamma_{k}\left(F_i(x_{k,i}) + \eta_k\nabla f_i\left(x_{k,i}\right)\right)^T\left(x^*_{\eta_k}- x_{k,i}\right) .
	\end{align*}
From the strong monotonicity of  ${{\nabla}} f_i  $ and the monotonicity of $ F_i $ we can write
\begin{align*}
&2 \gamma_{k}\left(F_i(x_{k,i}) + \eta_k\nabla f_i\left(x_{k,i}\right)\right)^T\left(x^*_{\eta_k}- x_{k,i}\right)\\
& \leq 2\gamma_{k}F_i\left(x^*_{\eta_k}\right)^T\left(x^*_{\eta_k}- x_{k,i}\right) \\
&+{2\gamma_{k}\eta_k\left(\nabla f_i\left(x^*_{\eta_k}\right)^T\left(x^*_{\eta_k}- x_{k,i}\right) -  \mu_{\text{min}}\left\|x^*_{\eta_k}- x_{k,i}\right\|^2\right)}.
\end{align*}
From the preceding two relations we obtain
\begin{align*}
	&\left\| x_{k,i+1} - x^*_{\eta_k} \right\|^2 
	\leq  \left(1- 2\gamma_k\eta_k \mu_{\text{min}}\right)\left\| x_{k,i} - x^*_{\eta_k} \right\|^2\\
	&+\gamma_k^2\left(\mytfrac{C_F+\eta_kC_f}{m}\right)^2 +2\gamma_{k}F_i\left(x^*_{\eta_k}\right)^T\left(x^*_{\eta_k}- x_{k,i}\right)\nonumber \\
	& +2\gamma_{k}\eta_k\nabla f_i\left(x^*_{\eta_k}\right)^T\left(x^*_{\eta_k}- x_{k,i}\right) .
\end{align*}
Adding and subtracting $2\gamma_kF_i (x^*_{\eta_{{k}}} )^Tx_k + 2\gamma_k\eta_k\nabla f_i (x^*_{\eta_{{k}}} )^Tx_k$ in the previous relation we get
{\begin{align*}
	&\left\| x_{k,i+1} - x^*_{\eta_k} \right\|^2 \leq  \left(1- 2\gamma_k\eta_k \mu_{\text{min}}\right)\left\| x_{k,i} - x^*_{\eta_k} \right\|^2\nonumber \\ 
	& +\gamma_k^2 \left(\mytfrac{C_F+\eta_kC_f}{m}\right)^2 +2\gamma_{k}F_i\left(x^*_{\eta_k}\right)^T\left(x^*_{\eta_k}- x_{k}\right) \nonumber\\ 
	&+2\gamma_{k}\left(\eta_k\nabla f_i\left(x^*_{\eta_k}\right)^T\left(x^*_{\eta_k}- x_{k}\right) + \left|F_i\left(x^*_{\eta_k}\right)^T\left(x_{k}- x_{k,i}\right)\right|\right) \\
	& +2\gamma_{k}  \eta_k \left|\nabla f_i\left(x^*_{\eta_k}\right)^T\left(x_{k}- x_{k,i}\right) \right| .
\end{align*}}
Employing the Cauchy-Schwarz inequality we obtain
\begin{align*}
&\left\| x_{k,i+1} - x^*_{\eta_k} \right\|^2 \leq  \left(1- 2\gamma_k\eta_k \mu_{\text{min}}\right)\left\| x_{k,i} - x^*_{\eta_k} \right\|^2\\
&+\gamma_k^2 \left(\mytfrac{C_F+\eta_kC_f}{m}\right)^2+2\gamma_{k}F_i\left(x^*_{\eta_k}\right)^T\left(x^*_{\eta_k}- x_{k}\right)\nonumber \\ 
& +2\gamma_{k}\eta_k\nabla f_i\left(x^*_{\eta_k}\right)^T\left(x^*_{\eta_k}- x_{k}\right)+\mytfrac{2\gamma_{k}\left(C_F+\eta_{{k}}C_f\right)}{m}\left\|x_{k}- x_{k,i}\right\| .
\end{align*}
Next we take summations over {$i \in [m]$} from both sides. {Recall  $f(x)\triangleq\sum_{i = 1}^mf_i(x)$ and $F(x)\triangleq\sum_{i = 1}^mF_i(x)$.} Using  Definition \ref{def:convenient_notation} for $x_{k,1}$ and recalling  $1-2\gamma_k\eta_k\mu_{\text{min}}< 1$  we have
\begin{align}\label{eqn:iterative_bound_3}
	&\sum_{i = 1}^m\left\| x_{k,i+1} - x^*_{\eta_k} \right\|^2 \leq  \left(1- 2\gamma_k\eta_k \mu_{\text{min}}\right)\left\| x_{k} - x^*_{\eta_k} \right\|^2 \nonumber \\
	&+  \sum_{i = 2}^m\left\| x_{k,i} - x^*_{\eta_k} \right\|^2 +\gamma_k^2 
	\mytfrac{\left(C_F+\eta_kC_f\right)^2}{m}\nonumber\\
	&+\mytfrac{2\gamma_{k}\left(C_F+\eta_{{k}}C_f\right)}{m}\sum_{i = 1}^m \left\|x_{k}- x_{k,i}\right\|+2\gamma_{k}F\left(x^*_{\eta_k}\right)^T\left(x^*_{\eta_k}- x_{k}\right)\nonumber \\
	& +2\gamma_{k}\eta_k\nabla f\left(x^*_{\eta_k}\right)^T\left(x^*_{\eta_k}- x_{k}\right).
\end{align}
{From Lemma \ref{lem:neighbor_agents}, $\| x_k - x_{k,i} \|\leq (i-1)\gamma_{k}\left({C_F}+\eta_{k}C_f\right)/m$} for all {$i\in[ m]$}. Invoking this relation and Definition \ref{def:convenient_notation} we obtain
\begin{align*}
&	\left\| x_{k+1} - x^*_{\eta_k} \right\|^2 \leq \left(1- 2\gamma_k\eta_k\mu_{\text{min}}\right)\left\| x_{k} - x^*_{\eta_k} \right\|^2   +\gamma_k^2\left(C_F\right.\\
&\left.+\eta_kC_f\right)^2 + 2\gamma_k\underbrace{\left(F\left(x^*_{\eta_k}\right)+\eta_k\nabla f\left(x^*_{\eta_k}\right)\right)^T\left(x^*_{\eta_k}- x_{k}\right)}_{\text{term 1}}.
\end{align*}
From Definition \ref{def:regularized_problem}, $x^*_{\eta_k}$ is the solution to problem \eqref{prob:regularized_problem}.  Recalling  $x_k \in X$,  we have: term 1 $\leq 0$. We obtain 
\begin{align}\label{eqn:iterative_bound_4}
	\left\| x_{k+1} - x^*_{\eta_k} \right\|^2 &\leq 
	 \left(1- 2\gamma_k\eta_k\mu_{\text{min}}\right) \left\| x_{k} - x^*_{\eta_k} \right\|^2\nonumber\\& +\gamma_k^2\left(C_F+\eta_kC_f\right)^2.
\end{align}
{Next, consider the term $\| x_k-x^*_{\eta_k} \|^2$ as follows.}
{\begin{align*}
\left\| x_k-x^*_{\eta_k} \right\|^2 &= \left\| x_k-x^*_{\eta_{k-1}}\right\|^2+ \left\|x^*_{\eta_{k-1}}-x^*_{\eta_k} \right\|^2\\
&  + \underbrace{2\left(x_k-x^*_{\eta_{k-1}}\right)^T\left(x^*_{\eta_{k-1}}-x^*_{\eta_k}\right)}_{\text{term 2}}{.}
\end{align*}}
{Next, we  bound term 2 by recalling $2a^Tb\leq \|a\|^2/\alpha + \alpha\|b\|^2 $ where $a,b \in \mathbb{R}^n$ and $\alpha > 0$. For $\alpha := {1}/{\gamma_{k}\eta_k\mu_{\text{min}}}$, bounding term 2 in the preceding inequality we obtain}
\begin{align*}
	\left\| x_k-x^*_{\eta_k} \right\|^2 &=
	 \left(1+{\gamma_{k}\eta_k\mu_{\text{min}}}\right)\left\| x_k-x^*_{\eta_{k-1}} \right\|^2\\&+\left(1+\mytfrac{1}{\gamma_{k}\eta_k\mu_{\text{min}}}\right) \left\|x^*_{\eta_{k-1}}-x^*_{\eta_k} \right\|^2 .
\end{align*}
From Lemma \ref{lem:recursive_bound_reg_problem}(b) we obtain
\begin{align}\label{eqn:iterative_bound_5}
	&\left\| x_k-x^*_{\eta_k} \right\|^2 \leq \left(1+{\gamma_{k}\eta_k\mu_{\text{min}}}\right)\left\| x_k-x^*_{\eta_{k-1}} \right\|^2\nonumber\\&+\left(1+\mytfrac{1}{\gamma_{k}\eta_k\mu_{\text{min}}}\right)\mytfrac{C_f^2}{m^2\mu_{\text{min}}^2}\left(1-\mytfrac{\eta_{k}}{\eta_{k-1}}\right)^2.
\end{align}
From equations \eqref{eqn:iterative_bound_4} and \eqref{eqn:iterative_bound_5} we have
\begin{align*}
	&\left\| x_{k+1} - x^*_{\eta_k} \right\|^2\\
	& \leq 
	\left(1- 2\gamma_k\eta_k\mu_{\text{min}}\right) \left(1+{\gamma_{k}\eta_k\mu_{\text{min}}}\right)\left\| x_k-x^*_{\eta_{k-1}} \right\|^2\\
	&+\left(1+\mytfrac{1}{\gamma_{k}\eta_k\mu_{\text{min}}}\right)\mytfrac{C_f^2}{m^2\mu_{\text{min}}^2}\left(1-\mytfrac{\eta_{k}}{\eta_{k-1}}\right)^2 +\gamma_k^2\left(C_F+\eta_kC_f\right)^2.
\end{align*}
{Using $0<\gamma_k\eta_k\mu_{\text{min}}\leq 0.5$ we have the desired result.}
\end{proof}
\subsection{Convergence analysis}\label{subsec:convergence}
{In this section, our goal is to derive a {non-asymptotic} convergence rate statement that relates the generated sequences by Algorithm \ref{alg:IR-IG_avg} to the Tikhonov trajectory. We begin with providing a class of sequences for the stepsize and the regularization parameter and prove some properties for them that will be used in the analysis. 
\begin{definition}[{Stepsize and regularization parameter}]\label{def:sequences_tikh_section}
Let $\gamma_k:=\mytfrac{\gamma}{(k+\Gamma)^a}$ and $\eta_k:=\mytfrac{\eta}{(k+\Gamma)^b}$ for all $k\geq 0$ where $\gamma,\eta,\Gamma, a$ and $b$ are strictly positive scalars. Let $a>b$, $a+b<1$, and $3a+b<2$. Assume that $\Gamma \geq 1$ and it is sufficiently large such that $\Gamma^{a+b}\geq 2\gamma\eta\mu_{\text{min}}$ and $\Gamma^{1-a-b}\geq  \mytfrac{4}{\gamma\eta\mu_{\text{min}}}$. 
\end{definition}
\begin{lemma}\label{lem:stepsize_reg_parameter_sequence_a} Consider Definition \ref{def:sequences_tikh_section}. The following results hold.
	
	\noindent(i) $\{\gamma_{k}\}$ and $\{\eta_{k}\}$ are strictly positive and nonincreasing such that $\gamma_0\eta_{{0}}\mu_{\text{min}} \leq  0.5$.
	
	\noindent(ii)  For all integers $k_1$ and $k_2$ such that $k_2\geq k_1\geq 0$ we have $1-\mytfrac{\eta_{k_2}}{\eta_{k_1}}  \leq \mytfrac{k_2-k_1}{k_2+\Gamma}$.
	
		\noindent(iii)  For all $k\geq 1$ we have $ \mytfrac{1}{\gamma_k^3\eta_{k}} \left( 1- \mytfrac{\eta_{{k}}}{\eta_{k-1}} \right)^2  \leq  \mytfrac{1}{\gamma^3\eta\Gamma^{2-3a-b}}$.
	
	\noindent(iv) For all $k\geq 1$ we have $\mytfrac{\gamma_{k-1}}{\eta_{k-1}} \leq \mytfrac{\gamma_{k}}{\eta_{k}}(1+0.5 \gamma_{k}\eta_{k}\mu_{\text{min}})$.
\end{lemma}}
{\begin{proof}
	See Appendix \ref{app:lem:stepsize_reg_parameter_sequence_a}.
\end{proof}}
{\fy{In Theorem \ref{thm:part2},} we derive agent-specific rates relating the sequences \fy{$\{x_{k,i}\}$} with the Tikhonov trajectory.
\begin{theorem}[Comparison with the Tikhonov trajectory]\label{thm:part2}
Consider problem \eqref{prob:initial_problem}. Let Assumption \ref{assum:initial_prob}(b), (c), and Assumption \ref{assum:initial_problem_3} hold. Consider $\{x_k\}$ and $\{x^*_{\eta_{{k}}}\}$ given in Definitions \ref{def:convenient_notation} and \ref{def:regularized_problem}, respectively. Let the stepsize sequence $\{\gamma_{k}\}$ and the regularization sequence $\{\eta_{{k}}\}$ be given by Definition \ref{def:sequences_tikh_section}. Then for all $ k\geq 0$ and all {$i\in [m]$} we have 
\begin{align*}
 \left\|x_{k+1,i}  - x^*_{\eta_k}\right\|^2 \leq \mytfrac{2(i-1)^2\left(C_F+\eta_0C_f\right)^2\gamma^{2}}{m^2(k+\Gamma+1)^{2a}} +\mytfrac{2\tau B_0\gamma}{\mu_{\text{min}}\eta(k+\Gamma)^{a-b}},
\end{align*}
where $\tau \triangleq \text{max}\left\{ \mu_{\text{min}}\eta\gamma^{-1}B_0^{-1}\Gamma^{a-b}\left\|x_1-x^*_{\eta_0}\right\|^2, 2 \right\}$ and $B_0 \triangleq\mytfrac{1.5C_f^2}{m^2\mu_{\text{min}}^3\gamma^3\eta\Gamma^{2-3a-b}} + \left(C_F+\eta_0C_f\right)^2$.  
\end{theorem}}
\begin{proof}
{Consider  \eqref{eqn:iterative_bound_0}. From Lemma \ref{lem:stepsize_reg_parameter_sequence_a}, for  $k\geq 1$ we have}
\begin{align*}
&\left\| x_{k+1} - x^*_{\eta_k} \right\|^2 \leq \left(1 - \gamma_k \eta_{{k}}\mu_{\text{min}}\right)\left\| x_k - x^*_{\eta_{k-1}} \right\|^2\\&+ \mytfrac{1.5C_f^2\gamma_{k}^2}{m^2\mu_{\text{min}}^3\gamma^3\eta\Gamma^{2-3a-b}} + \gamma_k^2\left(C_F+\eta_0C_f\right)^2.
\end{align*}
Let us define the terms $ v_k {\triangleq} \left\| x_k-x^*_{\eta_{k-1}} \right\|^2$, $\alpha_{k} {\triangleq} \gamma_{k}\eta_{k}\mu_{\text{min}}$, and $\beta_k\triangleq B_0\gamma_k^2$ for $k \geq 1$. Therefore, for all $k \geq 1$ we have
\begin{align}\label{eqn:proposition2_1}
v_{k+1} \leq (1-\alpha_k)v_k + \beta_k.
\end{align}
From Lemma \ref{lem:stepsize_reg_parameter_sequence_a} (iii), for all $k\geq 1$ we have 
{\begin{align}\label{eqn:proposition2_2}
\mytfrac{\beta_{k-1}}{\alpha_{k-1}} \leq  \mytfrac{B_0\gamma_{k}}{\mu_{\text{min}}\eta_{k}}(1+0.5 \gamma_{k}\eta_{k}\mu_{\text{min}}) =\mytfrac{\beta_{k}}{\alpha_{k}}(1+0.5 \alpha_{k}) .
\end{align}}
Next, we show that $v_{k+1} \leq \tau \mytfrac{\beta_{k}}{\alpha_{k}}$ for all $k\geq 0$. We apply induction on $k\geq 0$. Note that this relation holds for $k:=0$ as an implication of the definition of $\tau$. Suppose $v_k \leq \tau\mytfrac{\beta_{k-1}}{\alpha_{k-1}}$ holds for some $k \geq 1$. From \eqref{eqn:proposition2_1} we obtain $v_{k+1} \leq (1-\alpha_k)\mytfrac{\beta_{k-1}}{\alpha_{k-1}} \tau + \beta_k$. Using the upper bound for the right-hand side given by \eqref{eqn:proposition2_2} we have
\begin{align*}
&v_{k+1} \leq \tau(1-\alpha_k)(1+0.5 \alpha_{k})\mytfrac{\beta_{k}}{\alpha_{k}}  + \beta_k	\\
&=  \tau(1-\alpha_{k} + 0.5 \alpha_k -0.5 \alpha_k^2)\mytfrac{\beta_{k}}{\alpha_{k}}  + \beta_k
= \tau\mytfrac{\beta_{k}}{\alpha_{k}} \\&-\tau(1-0.5)\beta_{k} -0.5\tau\alpha_{k}\beta_{k} +\beta_{k}\leq\tau\mytfrac{\beta_{k}}{\alpha_{k}} +  (1-0.5\tau)\beta_{k}.
\end{align*}
From the definition of $\tau$, $\tau \geq 2$ implying that $1-0.5\tau\leq 0$. This completes the proof of induction. Recall from Lemma \ref{lem:neighbor_agents} that we have $\| x_k - x_{k,i} \|\leq (i-1)\gamma_{k}\left(C_F+\eta_{k}C_f\right)/m$ for any {$i \in [m]$}.  {For all $k \geq 0$ and  $i\in[m]$ we have}
\begin{align*}
 \left\|x_{k+1,i}  - x^*_{\eta_k}\right\|^2 &\leq 2 \left\| x_{k+1,i} -x_{k+1}\right\|^2+ 2\left\|x_{k+1} - x^*_{\eta_k}\right\|^2\\
 &\leq \mytfrac{2(i-1)^2\left(C_F+\eta_0C_f\right)^2\gamma_{k+1}^2}{m^2} +\mytfrac{2\tau B_0\gamma_{k}}{\mu_{\text{min}}\eta_k}\\
 & =\mytfrac{2(i-1)^2\left(C_F+\eta_0C_f\right)^2\gamma^2}{m^2(k+\Gamma+1)^{2a}} +\mytfrac{2\tau B_0\gamma}{\mu_{\text{min}}\eta(k+\Gamma)^{a-b}}.
\end{align*}
Hence the proof is completed. 
\end{proof}
\fy{\begin{remark} Recall that from Theorem~\ref{thm:rates}, the convergence rates of $\mathcal{O}(N^{-0.25})$ are obtained for $(a,b) =(0.5,0.25)$. These choices satisfy conditions in Def.~\ref{def:sequences_tikh_section}. In view of Theorem~\ref{thm:part2}, we have $ \left\|x_{k+1,i}  - x^*_{\eta_k}\right\|^2 \leq \mathcal{O}(k^{-0.25})$ for $i \in [m]$. 
\end{remark}}

\section{Numerical results} \label{sec:numerical_implementation}

\noindent {\bf (i) A traffic equilibrium problem:} For an illustrative example, we consider the transportation network in \cite{CZF2009NCP}. We first describe the network and present the NCP formulation. Then, we implement Algorithm \ref{alg:IR-IG_avg} to solve model \eqref{prob:dist_best_NCP} and compute the best equilibrium.
\begin{wrapfigure}{r}{4cm}
  \begin{center}
\includegraphics[width=3.9cm]{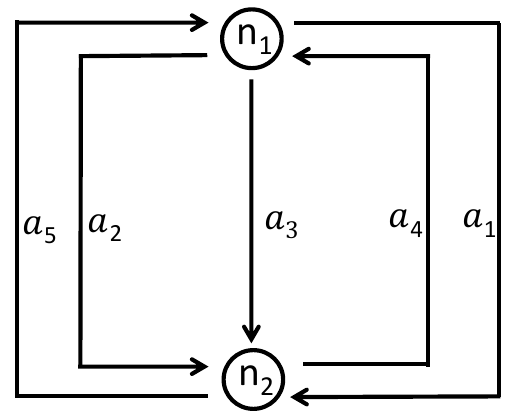}
  \end{center}
\caption{A transportation network with $2$ nodes and $5$ arcs}\label{fig:network}
\end{wrapfigure}
 Consider {a transportation network with the set of nodes $\{n_1,n_2\}$  and {the set of} directed arcs $ \{a_1,a_2,a_3,a_4,a_5\} $, \fy{as shown in Figure \ref{fig:network}.} 
Note that $a_1$ and $a_4$ construct a two-way road. The same holds for $a_2$ and $a_5$. We let $d \triangleq [d_1,d_2]^T$ denote the expected travel demand vector where  $d_1$ and $d_2$ correspond to the demand from $n_1$ to $n_2$, and from $n_2$ to $n_1$, respectively. Let the vector $h\triangleq [h_1,\ldots,h_5]^T$ denote the traffic flow on the arcs. The travel cost on arc $i$ \fy{is assumed to be} $[Ch+q]_i$ \fy{where the} cost matrix $C\in \mathbb{R}^{5\times 5}$ and vector $q\in \mathbb{R}^5$ be given by\begin{table*}[t]
	\renewcommand\thetable{2}
	\setlength{\tabcolsep}{0pt}
	\centering
	\begin{tabular}{c | c  c  c c}
		{\footnotesize$(\gamma_0,\eta_0) \ $}& {\footnotesize  $(0.1,0.1)$} & {\footnotesize $(0.1,1)$} &  {\footnotesize $(1,0.1)$} & {\footnotesize $(1,1)$}  \\
		\hline\\
		\rotatebox[origin=c]{90}{{\footnotesize $\ln\left(\text{infeasibility}\right)$}}
		&
		\begin{minipage}{.22\textwidth}
			\includegraphics[scale=.125, angle=0]{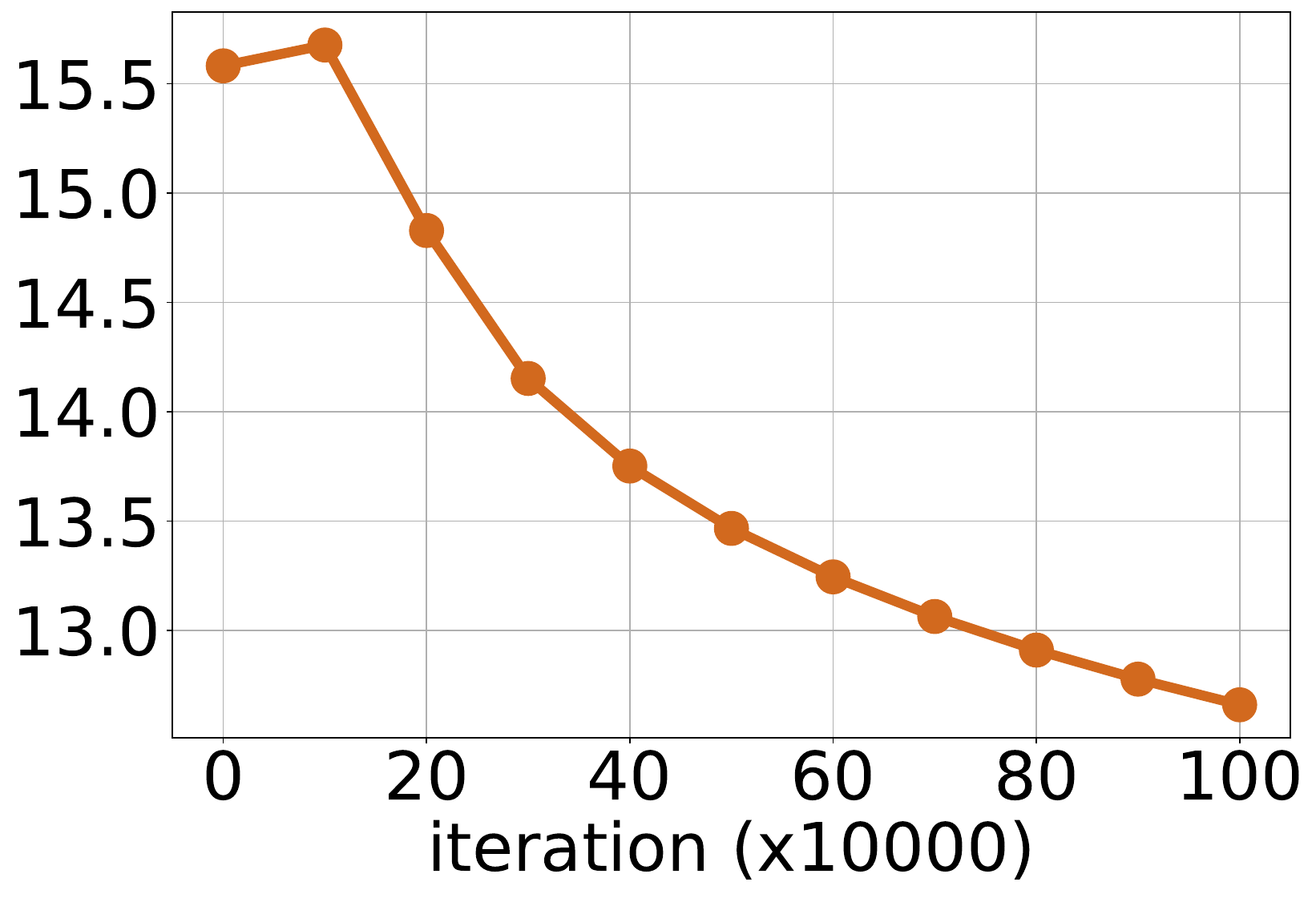}
		\end{minipage}
		&
		\begin{minipage}{.22\textwidth}
			\includegraphics[scale=.125, angle=0]{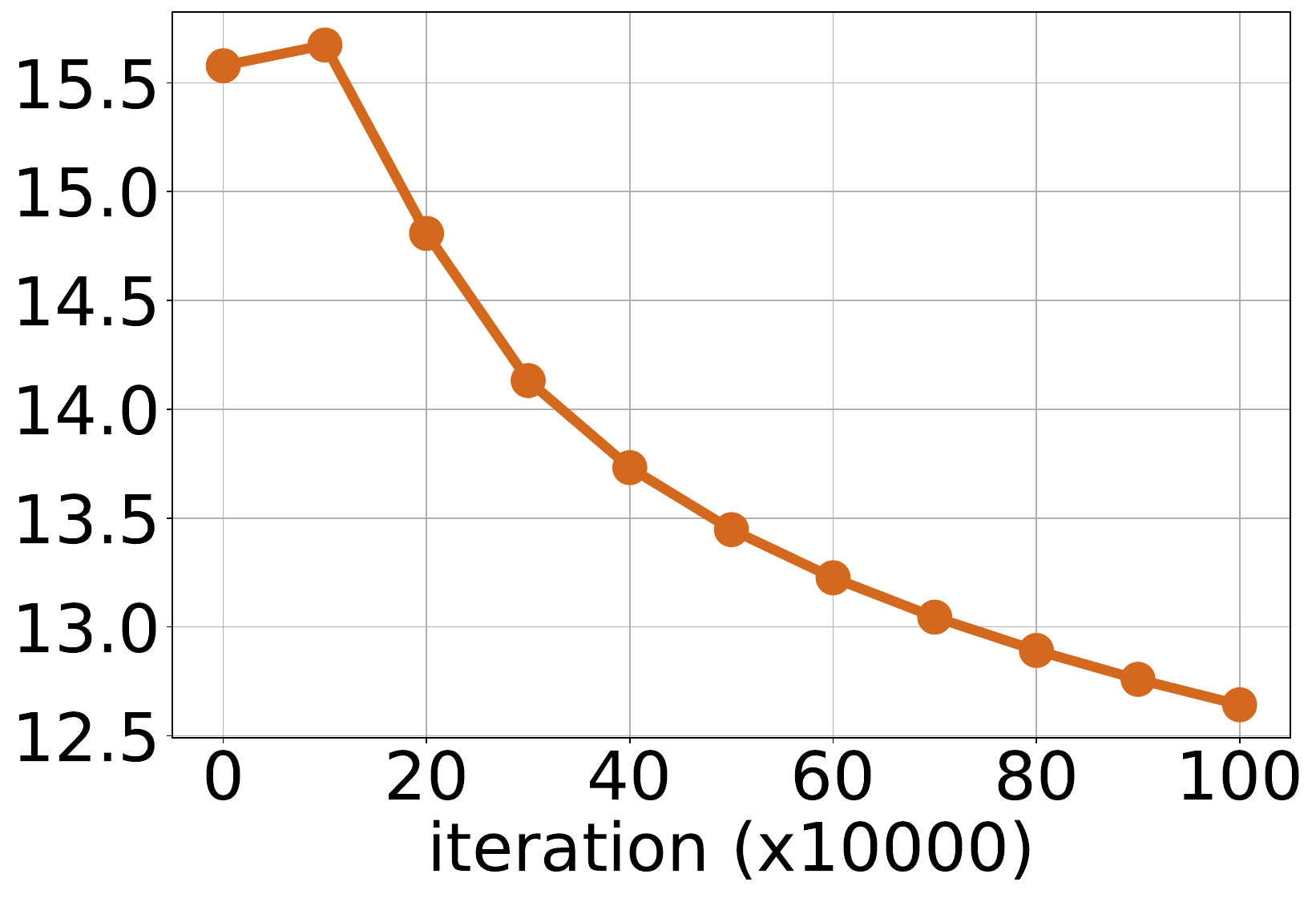}
		\end{minipage}
		&
		\begin{minipage}{.22\textwidth}
			\includegraphics[scale=.125, angle=0]{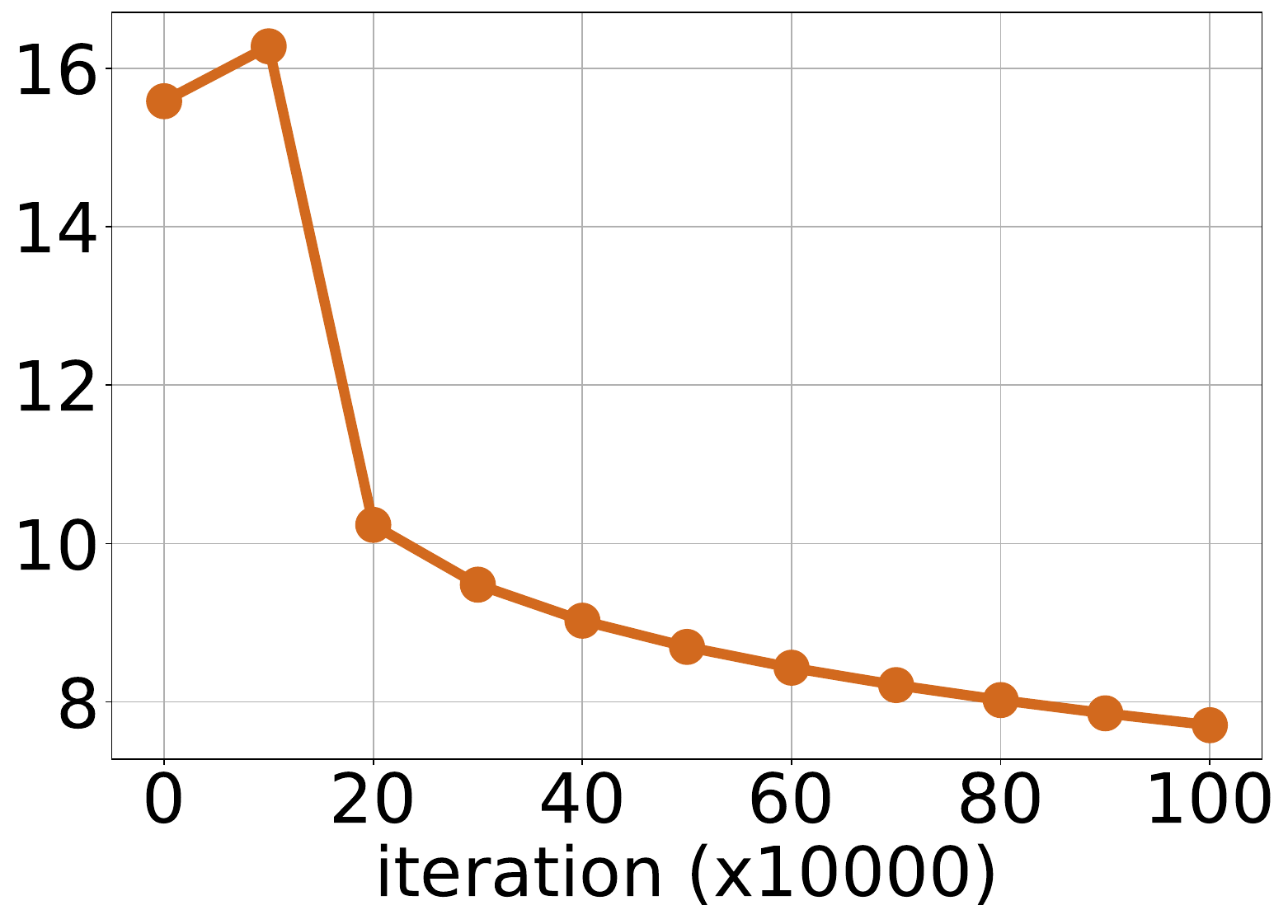}
		\end{minipage}
		&
		\begin{minipage}{.22\textwidth}
			\includegraphics[scale=.125, angle=0]{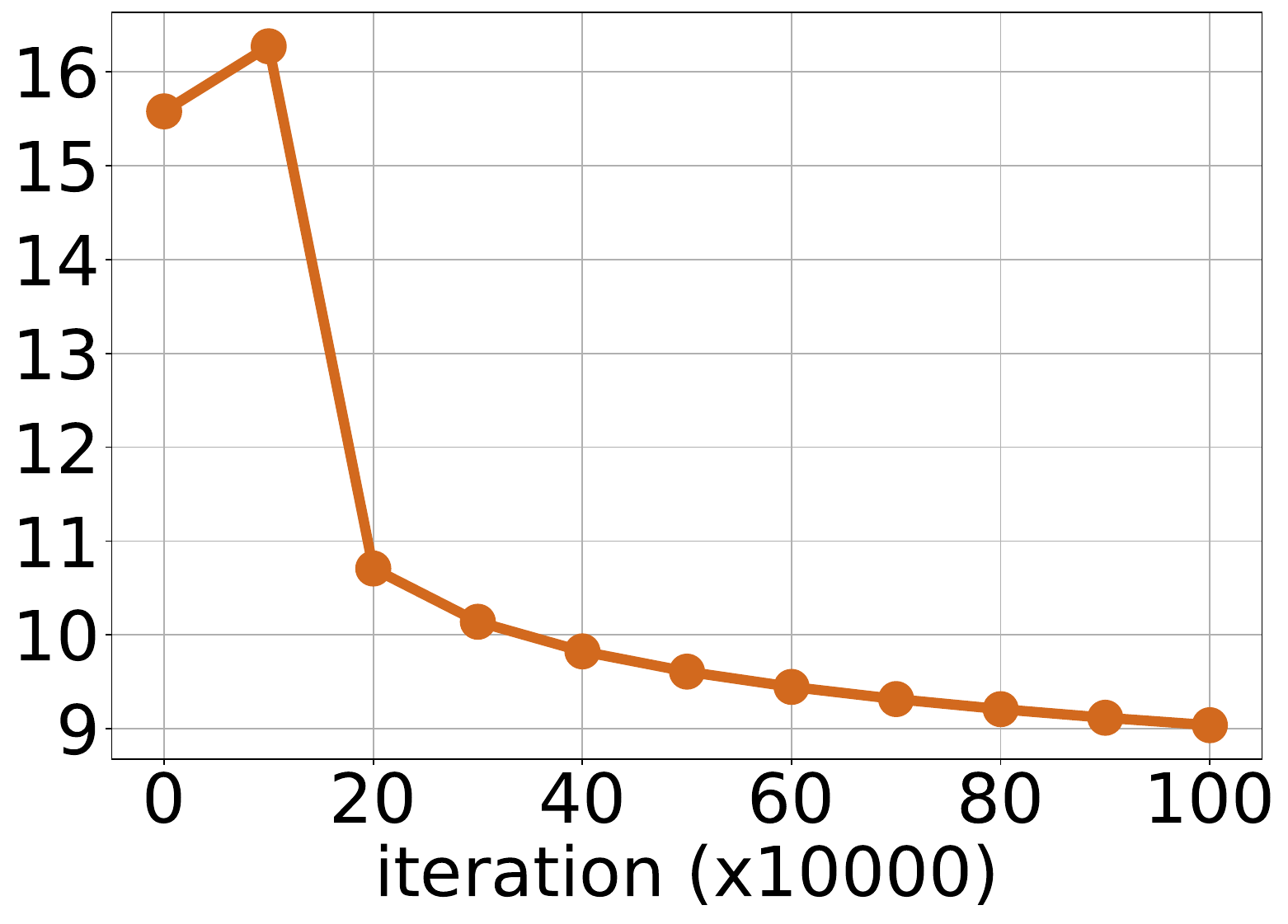}
		\end{minipage}
		\\
		\hbox{}& & & \\
		\hline\\
		\rotatebox[origin=c]{90}{{\footnotesize$\ln\left(\text{objective}\right)$}}
		&
		\begin{minipage}{.22\textwidth}
			\includegraphics[scale=.125, angle=0]{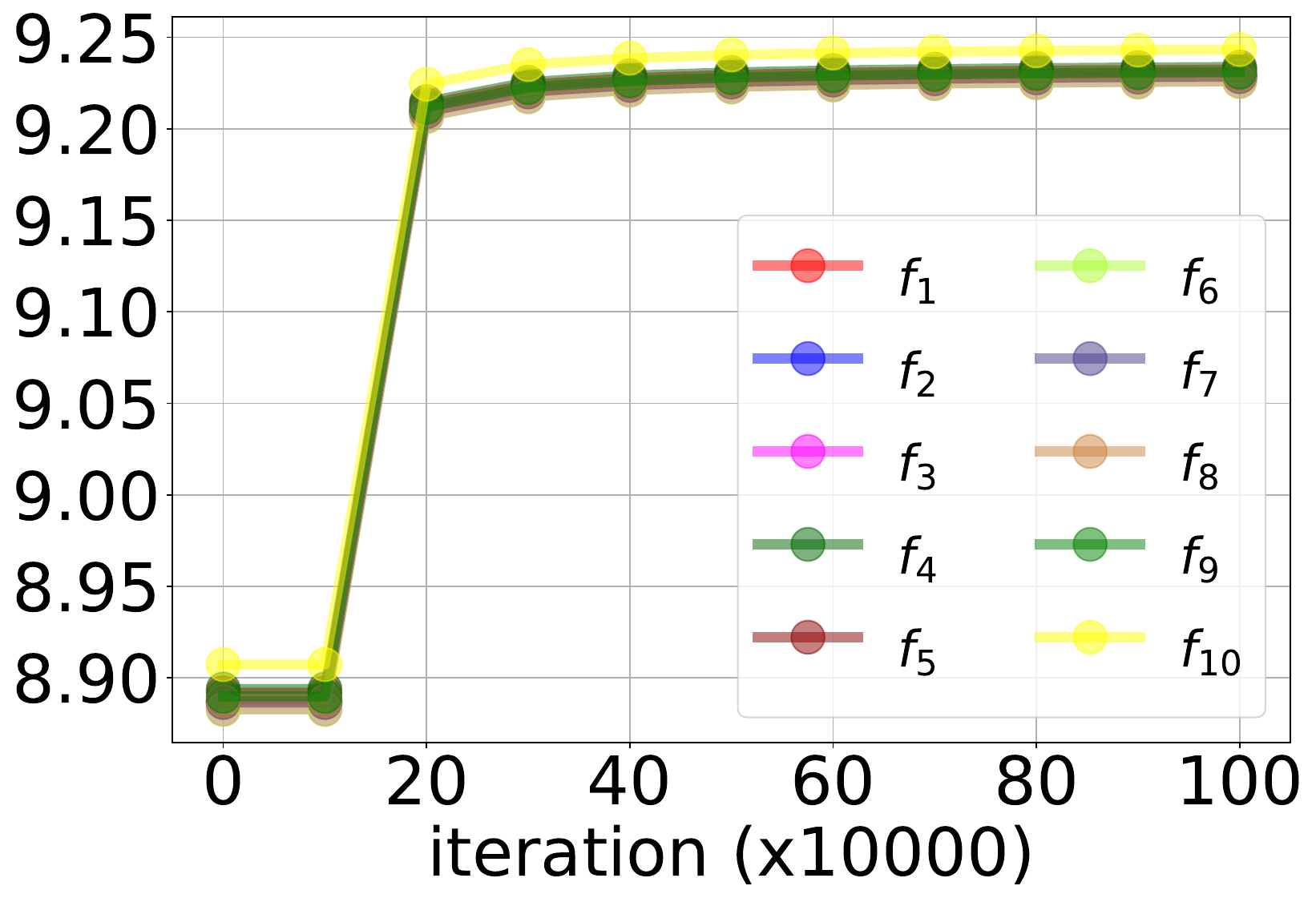}
		\end{minipage}
		&
		\begin{minipage}{.22\textwidth}
			\includegraphics[scale=.125, angle=0]{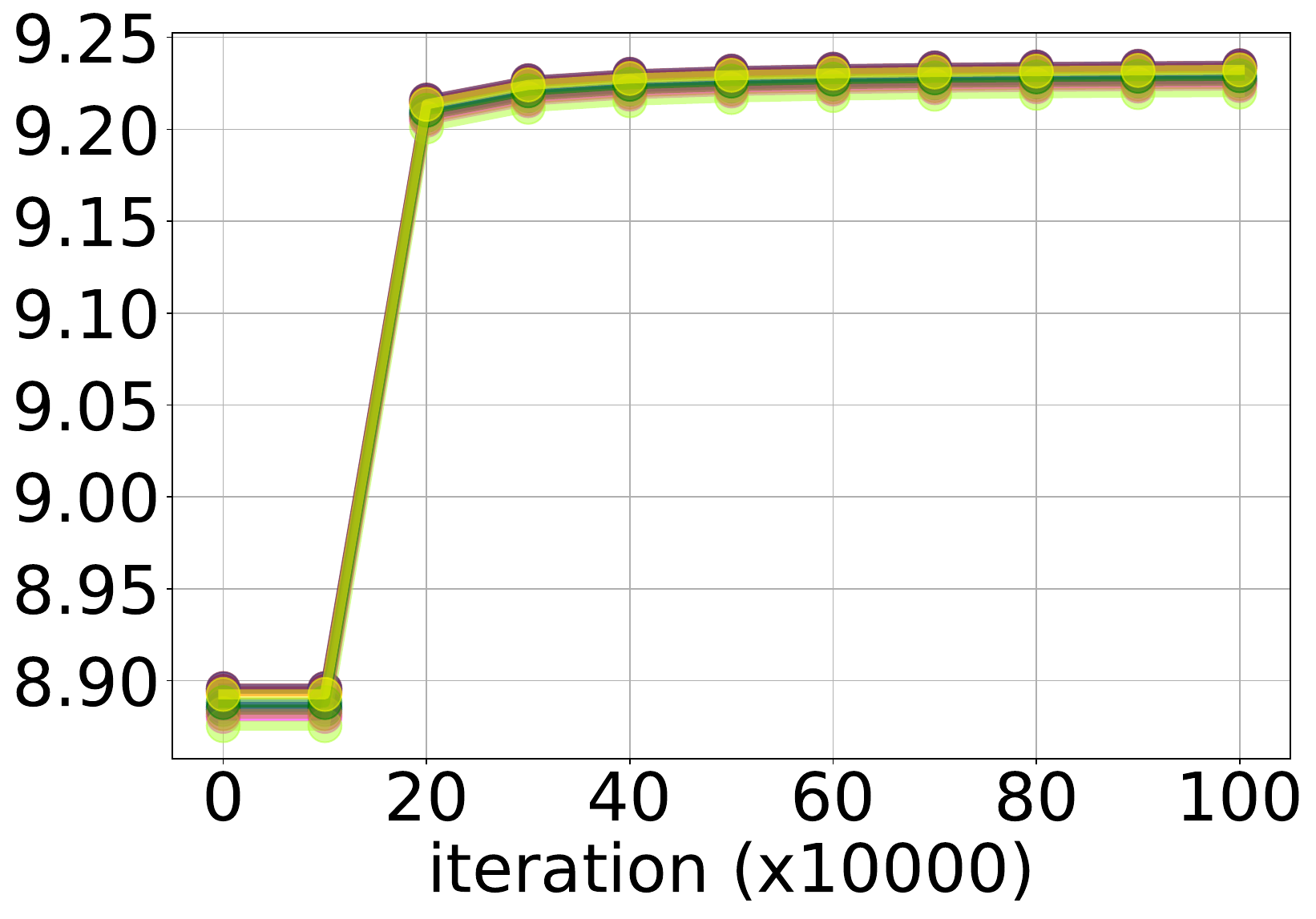}
		\end{minipage}
		&
		\begin{minipage}{.22\textwidth}
			\includegraphics[scale=.125, angle=0]{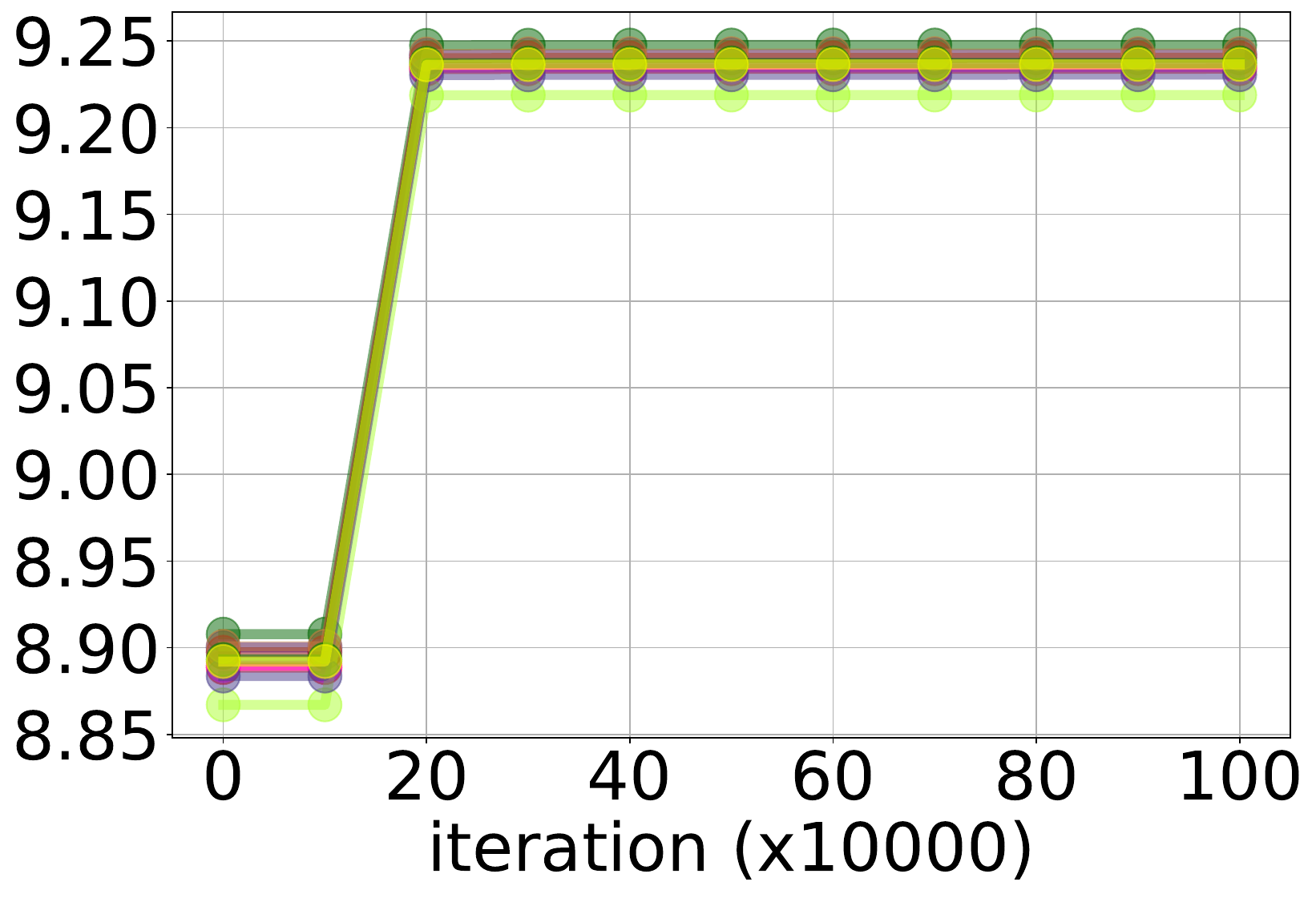}
		\end{minipage}
		&
		\begin{minipage}{.22\textwidth}
			\includegraphics[scale=.125, angle=0]{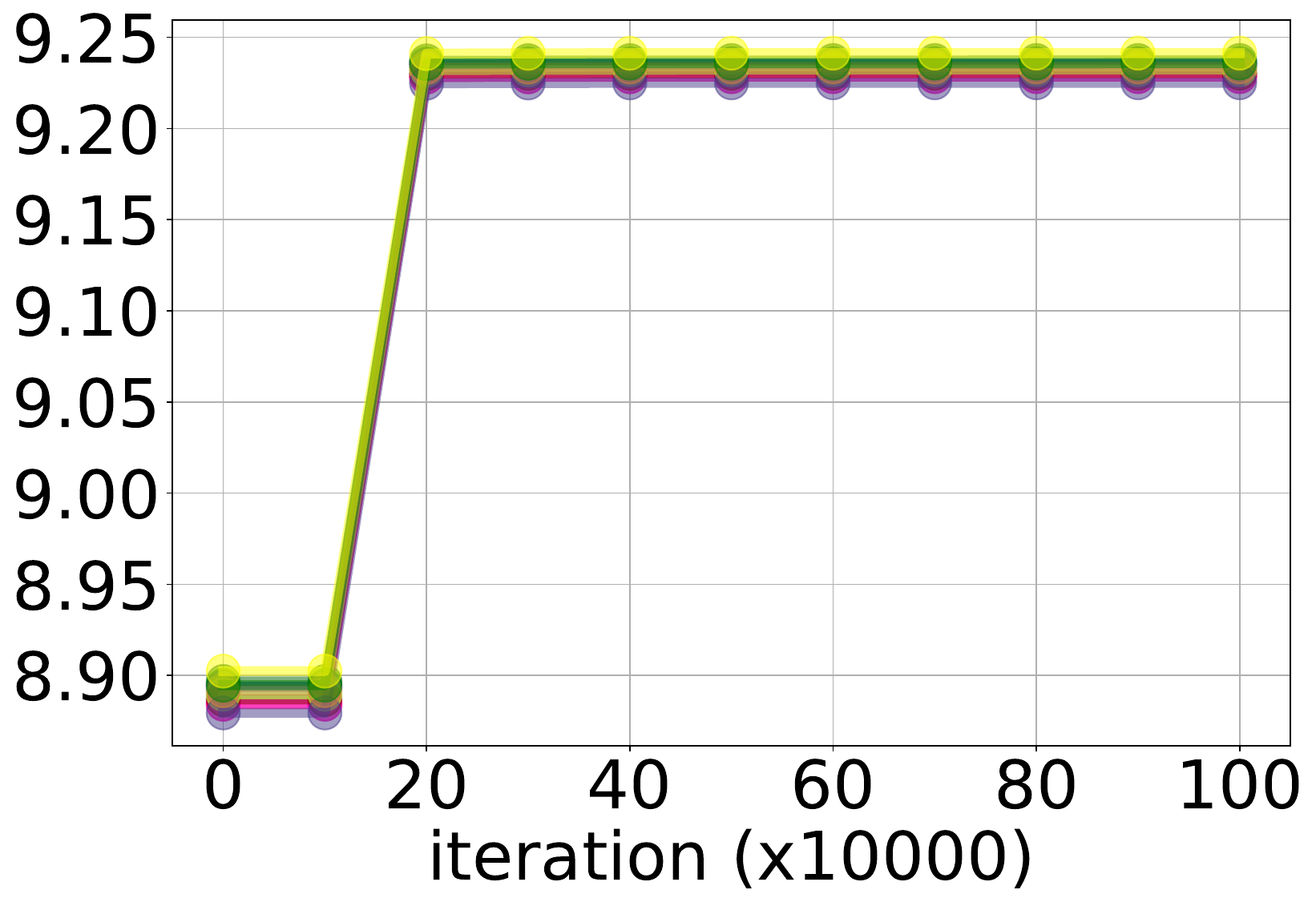}
		\end{minipage}
	\end{tabular}
	\captionof{figure}{\fy{Performance of Algorithm \ref{alg:IR-IG_avg} in addressing the transportation network example}}\label{fig:TEP}
\end{table*} 
 
  { \scriptsize \begin{align}\label{eqn:C_q}
C := \begin{bmatrix}
	0.92& 0&0&5&0\\0&5.92&0&0&5\\0&0&10.92&0&0\\ 2&0&0&10.92&0 \\ 0&1&0&0&15.92
\end{bmatrix}, \qquad q :=  \begin{bmatrix} 1000 \\ 950 \\ 3000 \\ 1000 \\ 1300 \end{bmatrix}.
\end{align}}

We note that the matrix $C$ is positive semidefinite. 
Intuitively speaking, the structure of $C$ implies that the cost of each arc in a two-way road depends on the flows on the both directions. Let $u\triangleq [u_1,u_2]^T$ denote the (unknown) vector of minimum travel costs between the origin-destination (OD) pairs, i.e., $u_1$ denotes the minimum travel cost from $n_1$ to $n_2$, and $u_2$ denotes the minimum travel cost from $n_2$ to $n_1$. The Wardrop user equilibrium principle represents the path choice behavior of the users based on the following rationale: (i) For any OD pair among all possible arcs, users tend to choose the {arc(s)} with the minimum cost. (ii) For any OD pair, the arc(s) that have the minimum cost will have positive flows and will have equal costs. (iii) For any OD pair, arcs with higher costs than the minimum value will have no flows. Mathematically the Wardrop's principle can be characterized as 
\begin{align}\label{eqn:wardrop}
	0 \leq C h + q - B^T u \perp h \geq 0, \quad 0 \leq Bh - d \perp u \geq 0,
\end{align} 
where $B \in \mathbb{R}^{2\times 5}$ denotes the (OD pair, arc)-incidence matrix given as {\scriptsize $B := \begin{bmatrix}1&1&1&0&0\\
0&0&0&1&1\end{bmatrix}$}. \fy{We} assume that the demand vector $d$ and the cost vector $q$ are subject to uncertainties \fy{and} define decision vector $x \in \mathbb{R}^7$, random variable \fy{$\xi \in \mathbb{R}^{7}$}, and stochastic mapping $F(\bullet,\xi):\mathbb{R}^7\to\mathbb{R}^7$ as 
\begin{align*}
x\triangleq {\scriptsize \begin{bmatrix}h \\ u\end{bmatrix}}, \qquad \xi \triangleq {\scriptsize \begin{bmatrix}\tilde d \\ \tilde q\end{bmatrix}}, \qquad F(x,\xi)\triangleq{\scriptsize \begin{bmatrix}C &-B^T \\ B& 0\end{bmatrix}\begin{bmatrix}h \\ u\end{bmatrix} +\begin{bmatrix}\tilde q \\ -\tilde d\end{bmatrix}}.
\end{align*}
Then from section \ref{sec:introduction}, the Wardrop equation \eqref{eqn:wardrop} can be characterized as $\mbox{VI}\left(\mathbb{R}_+^7,\mathbb{E}[F(\bullet,\xi)]\right)$. Notably due to positive semidefinite property of $C$, the mapping $\mathbb{E}[F(\bullet,\xi)]$ is merely monotone. Consequently, the aforementioned VI may have multiple equilibria. Among them, we seek to find the best equilibrium with respect to a welfare function $f$ defined as the expected total travel time over the network by all users, i.e., $f(x)\triangleq \mathbb{E}\left[(Ch+\tilde q)^T\mathbf{1}_5\right]$ \fy{where $\mathbf{1}_5$ is the 1-vector of size $5$.}
}

\noindent \textit{Set-up}: For this experiment, we assume that $\tilde d_1\sim \mathcal{N}(210,10)$, $\tilde d_2\sim \mathcal{N}(120,10)$. Also for $i=1,\ldots,5$ we let $\tilde q_i$ be normally distributed with the mean equal to $q_i$ and the standard deviation of $300$, where the vector $q$ is given by \eqref{eqn:C_q}. Following the formulation \eqref{prob:dist_best_NCP} we generate $1000$ samples for each parameter and distribute the data equally among $10$ agents. We let $\gamma_k:=\mytfrac{\gamma_0}{\sqrt{k+1}}$ and $\eta_k:=\mytfrac{\eta_0}{(k+1)^{0.25}}$ and consider different values for the initial stepsize $\gamma_{0}$ and the initial regularization parameter $\eta_{0}$. The results are as shown in Figure \ref{fig:TEP}. We use standard averaging by assuming that $r=0$. Notably for quantifying the infeasibility, we consider the metric $\phi(x) \triangleq \left\|\max\{0,-x\}\right\|^2 + \left\|\max\{0,-F(x)\}\right\|^2 + \left|x^TF(x)\right|$, where $F(x)\triangleq \sum_{i=1}^mF_i(x)$ and $F_i(x)\triangleq \sum_{\ell \in \mathcal{S}_i}F(x,\xi_\ell)$. Note that $\phi(x) =0$ if and only if $0\leq x\perp F(x) \geq 0$. We choose this metric over the dual gap function employed earlier in the analysis because in this particular example, the dual gap function becomes infinity at some of the evaluations of the generated iterates. This is due to the unboundedness of the set $X:=\mathbb{R}^n_+$. Unlike the dual gap function, $\phi(x)$ stays bounded and is more suitable to plot.

{\noindent \textit{Insights}:} In Figure \ref{fig:TEP} we observe that in all four different settings the infeasibility metric decreases as the algorithm proceeds. This indeed implies that the generated iterates by the agents tend to satisfy the {NCP constraints} with an increasing accuracy. In terms of the suboptimality metric we observe that the each agent's objective value  becomes more and more stable over time. Intuitively this implies that the agents asymptotically reach to an equilibrium.  We should note that although the function $f$ is minimized, it is minimized only over the set of equilibria. The fact that the objective values in Figure \ref{fig:TEP} are not necessarily decreasing is mainly because of the impact of feasibility violation of the iterates with respect to the {NCP constraints} throughout the implementations. 

\begin{table*}[t]
	\renewcommand\thetable{3}
	\centering
	\begin{tabular}{c|c|c}
		$ {|S|} \backslash {n}$  & 50 & 100 \\
		\hline
		100&\begin{minipage}{7.5cm}\centering\includegraphics[width=0.5\textwidth]{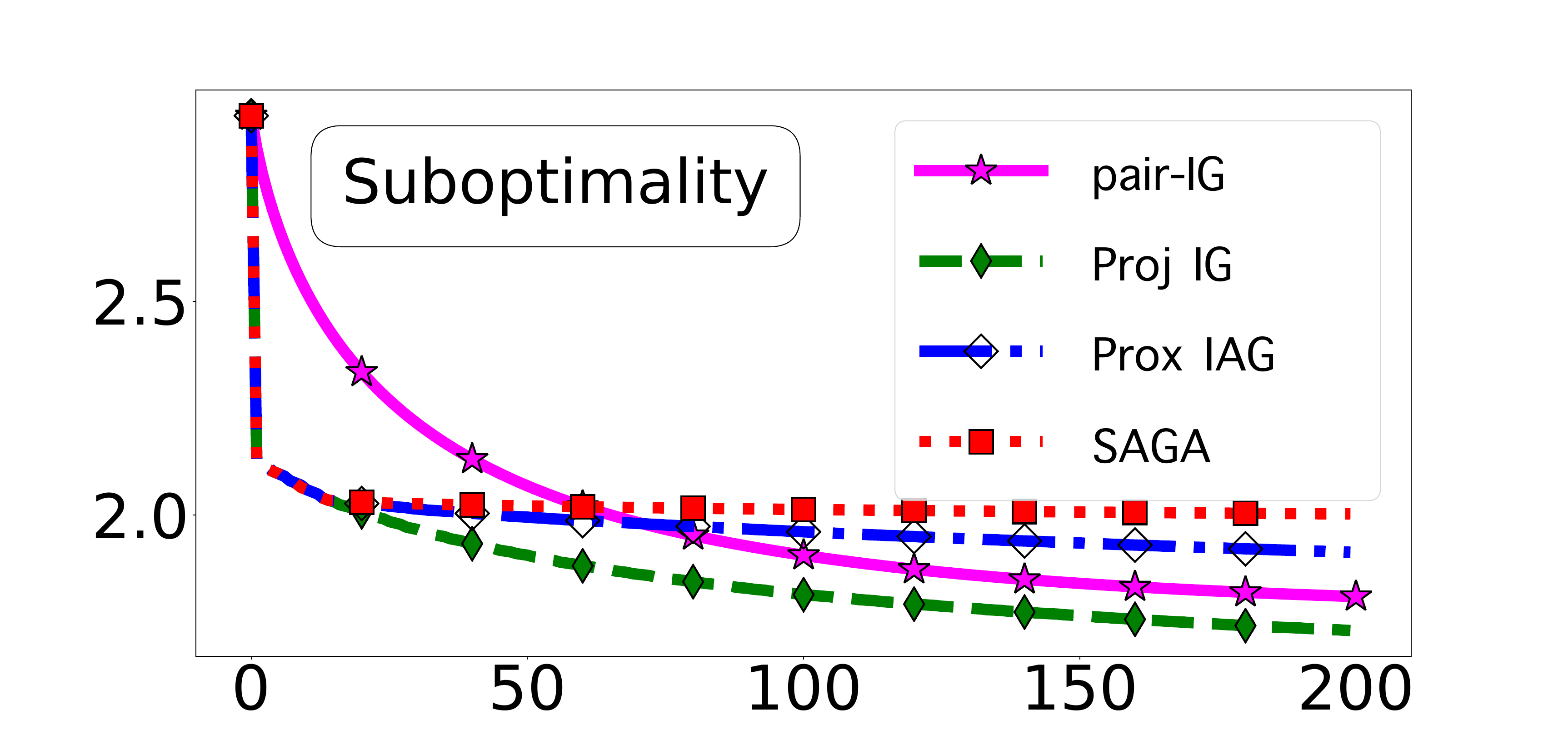}\includegraphics[width=0.5\textwidth]{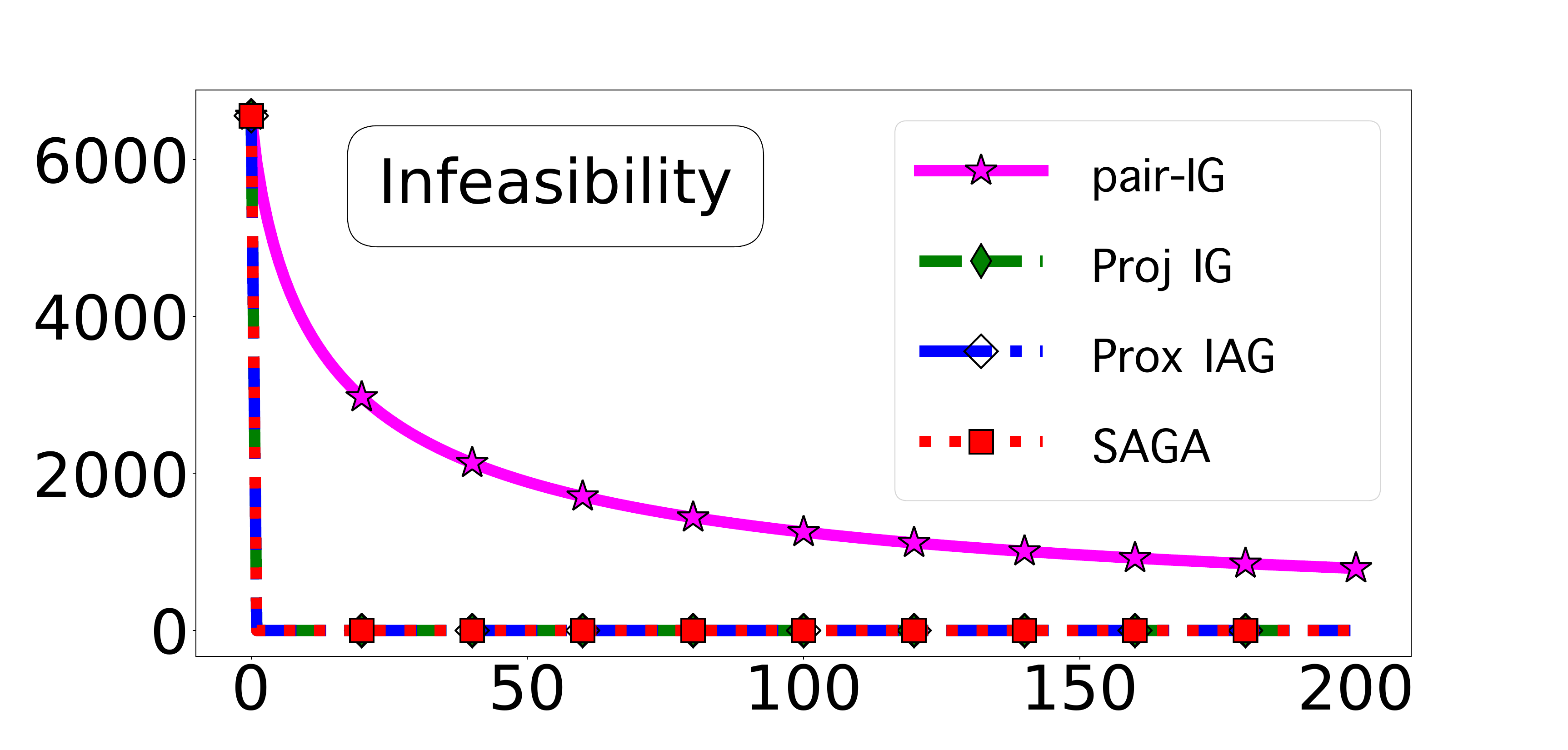}\end{minipage}
		&\begin{minipage}{7.5cm}\centering\includegraphics[width=0.5\textwidth]{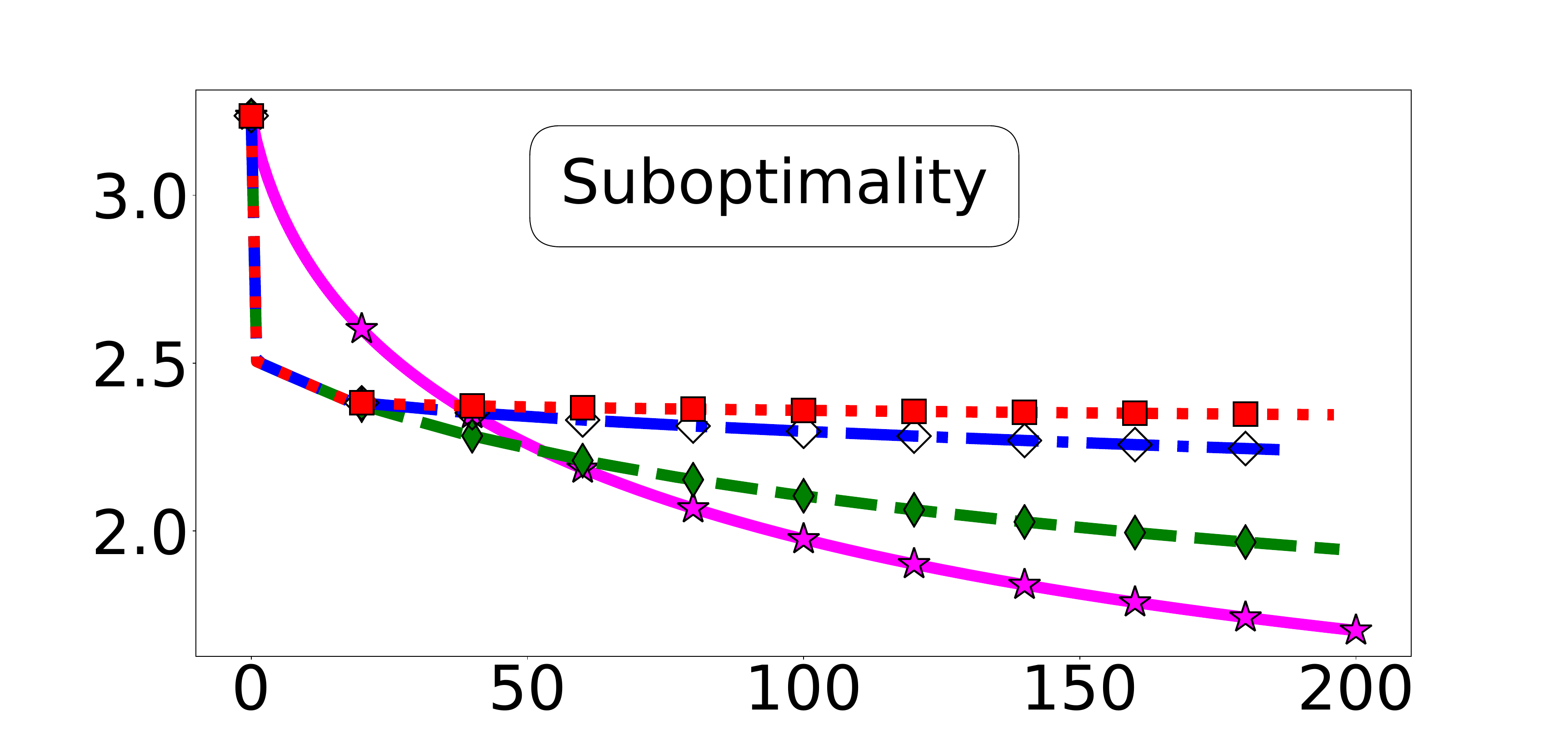}\includegraphics[width=0.5\textwidth]{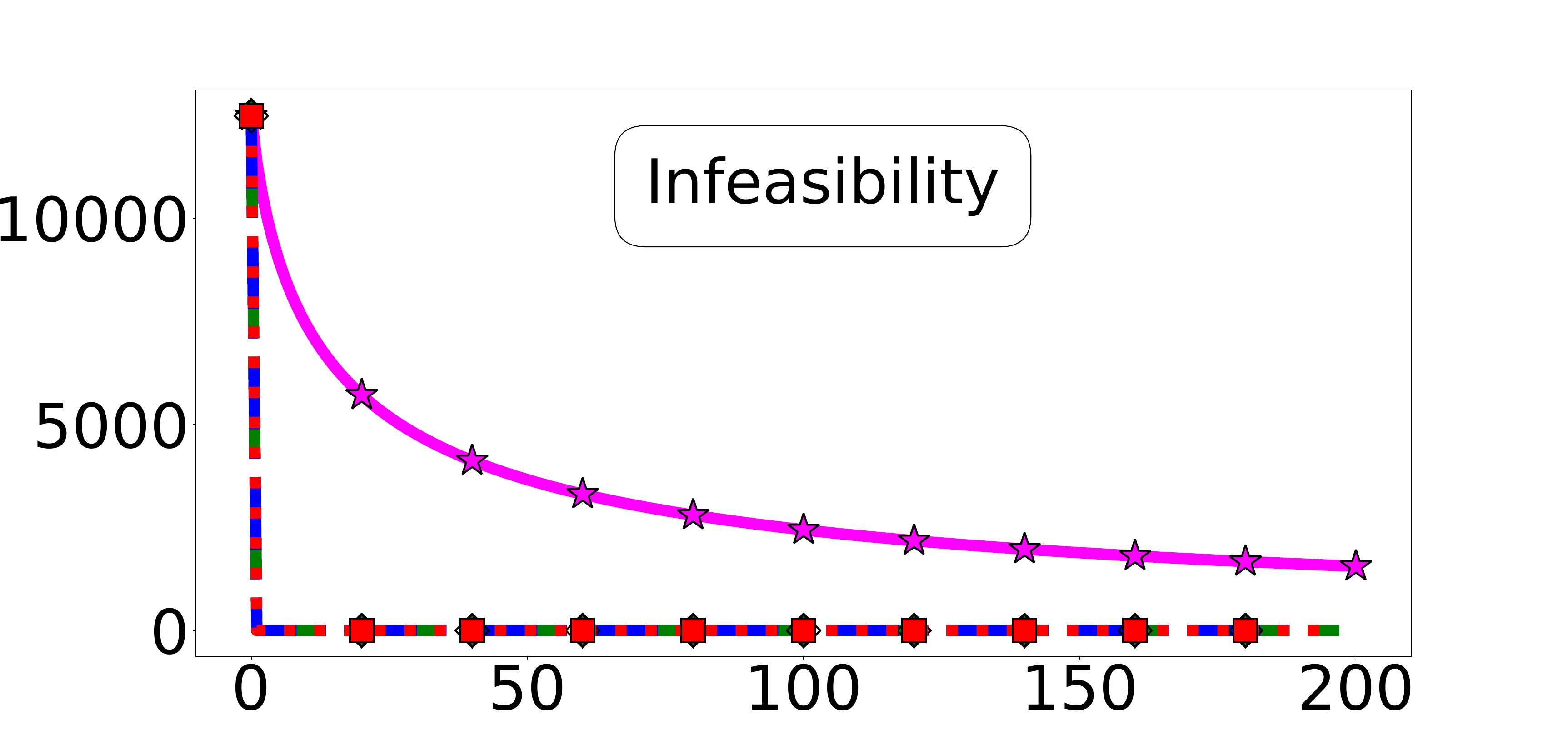}\end{minipage}
		\\\hline 200 &\begin{minipage}{7.5cm}\centering\includegraphics[width=0.5\textwidth]{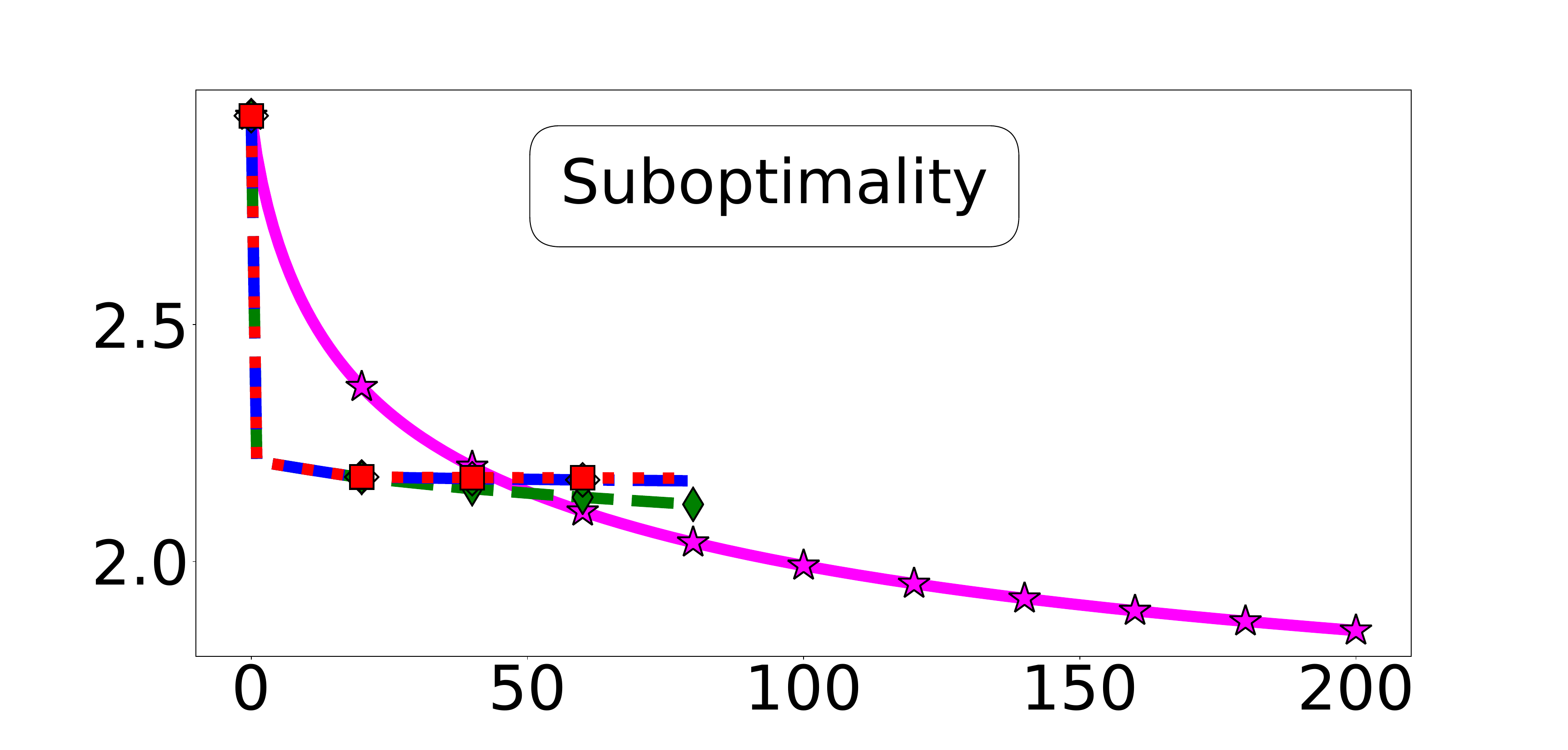}\includegraphics[width=0.5\textwidth]{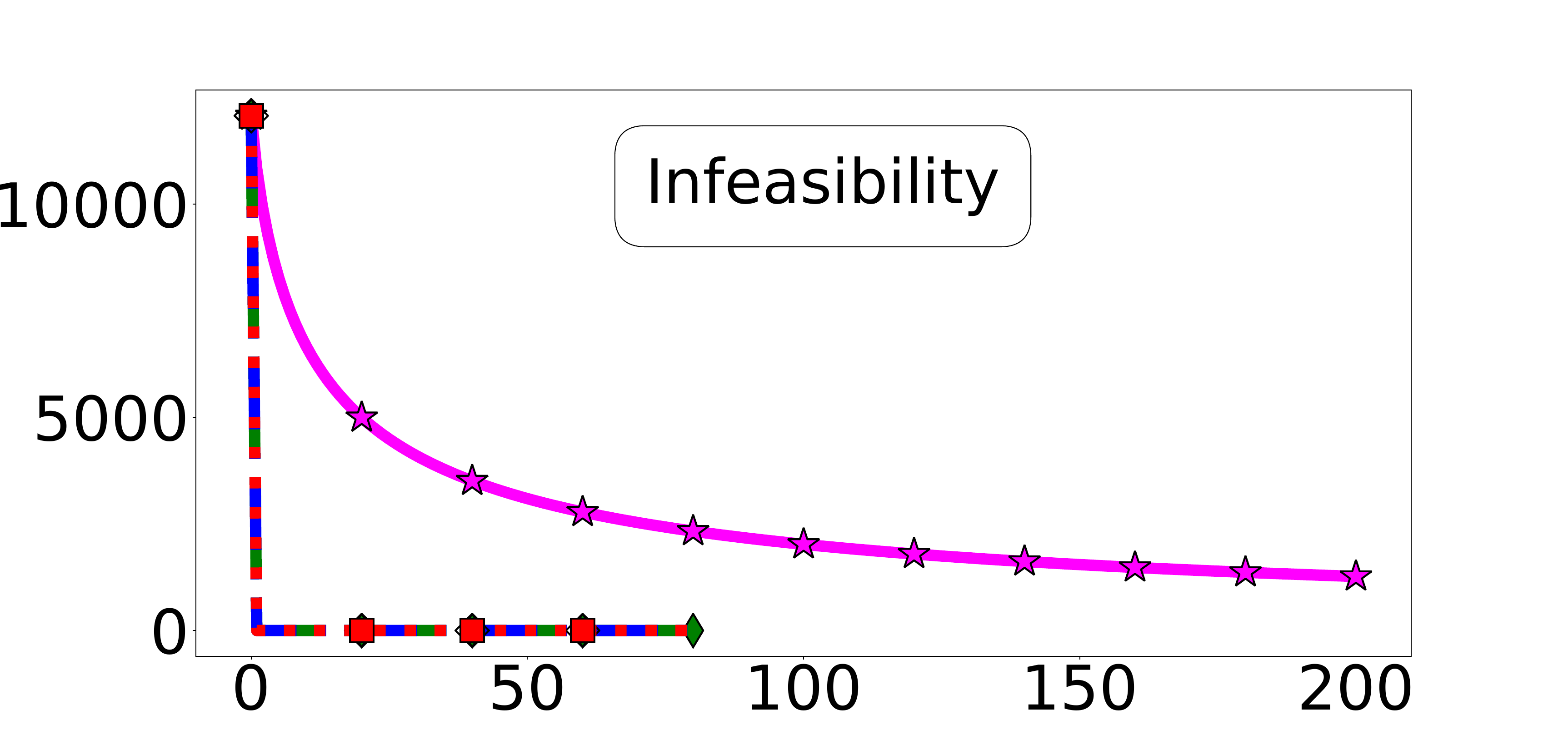}\end{minipage}
		&\begin{minipage}{7.5cm}\centering\includegraphics[width=0.5\textwidth]{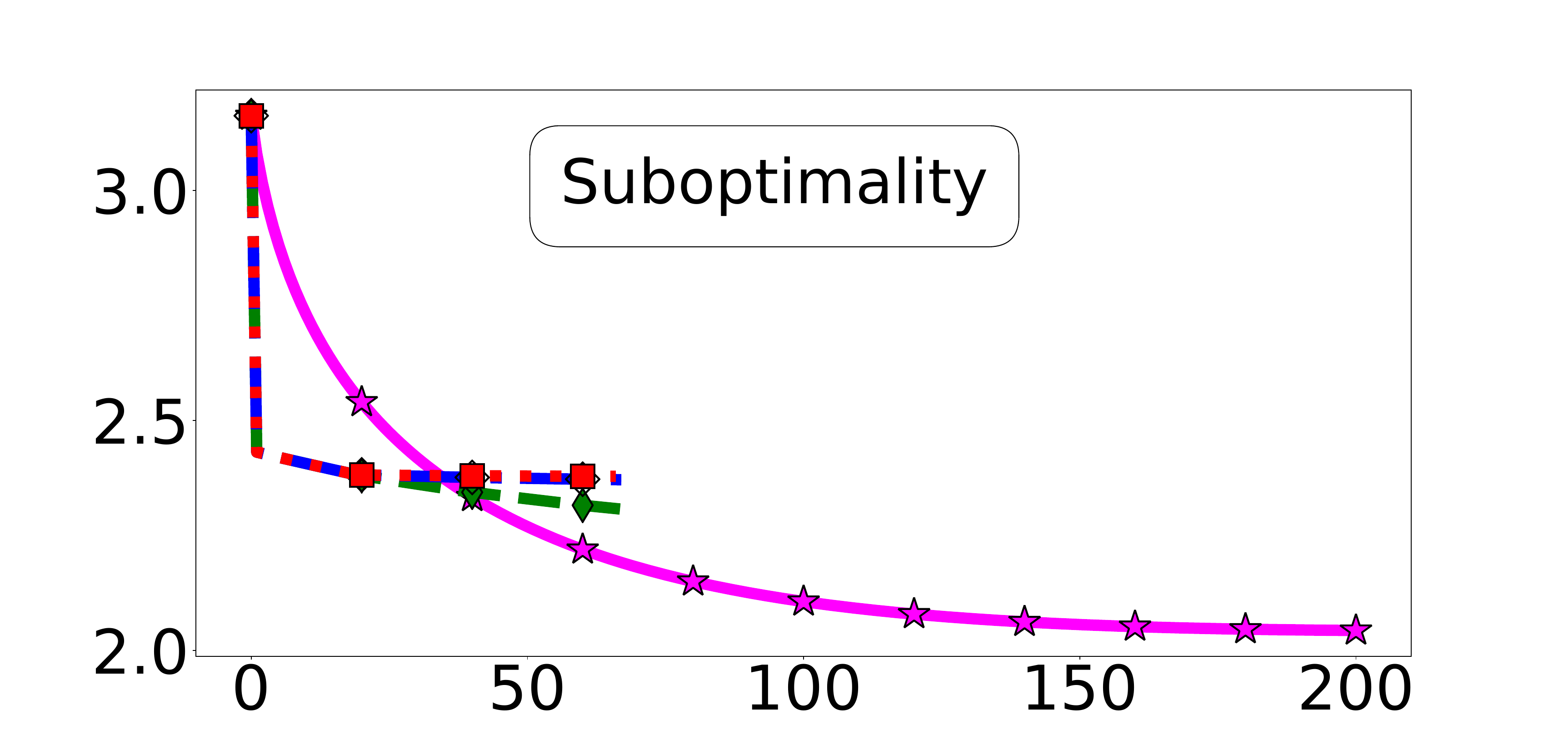}\includegraphics[width=0.5\textwidth]{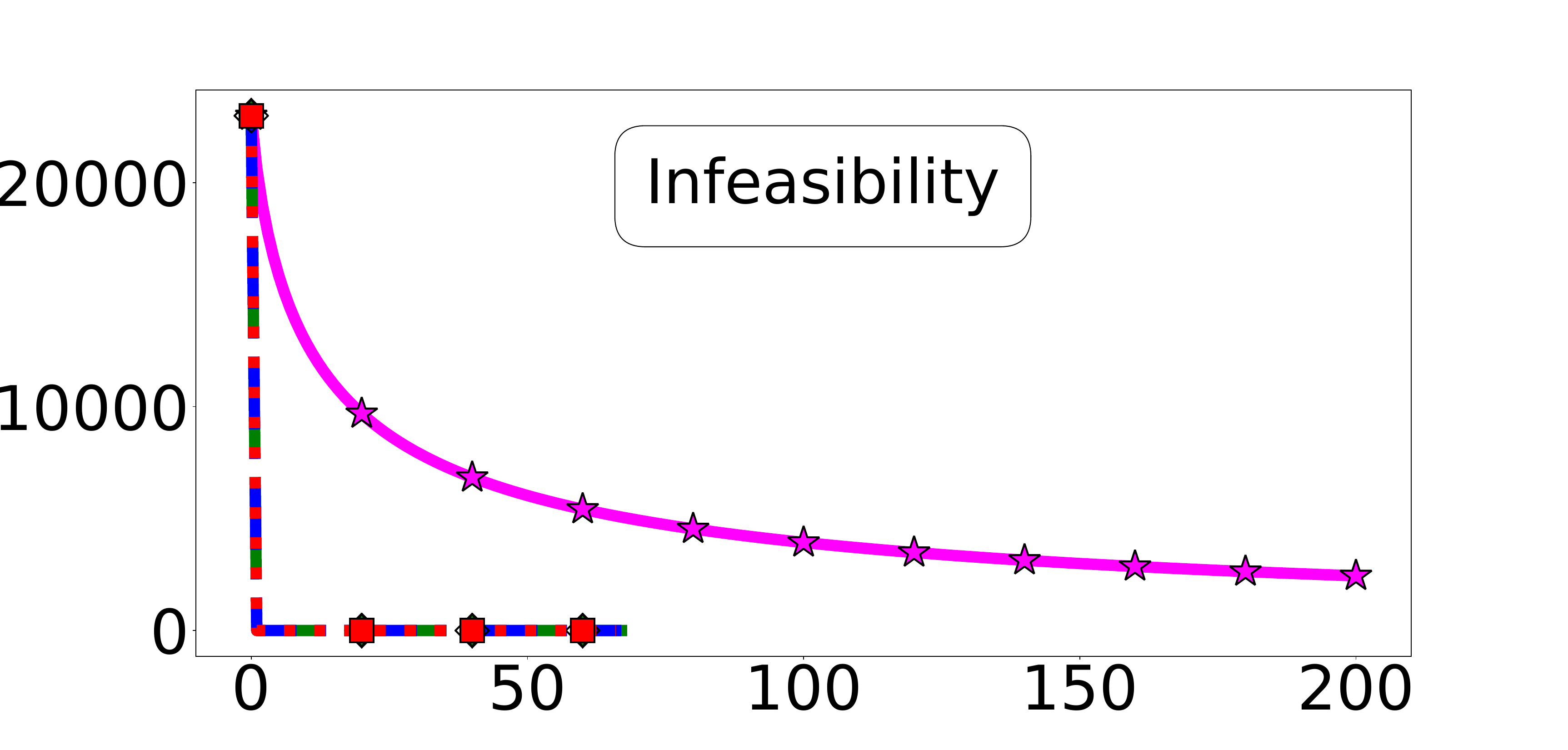}\end{minipage}
		\\\hline 500 &\begin{minipage}{7.5cm}\centering\includegraphics[width=0.5\textwidth]{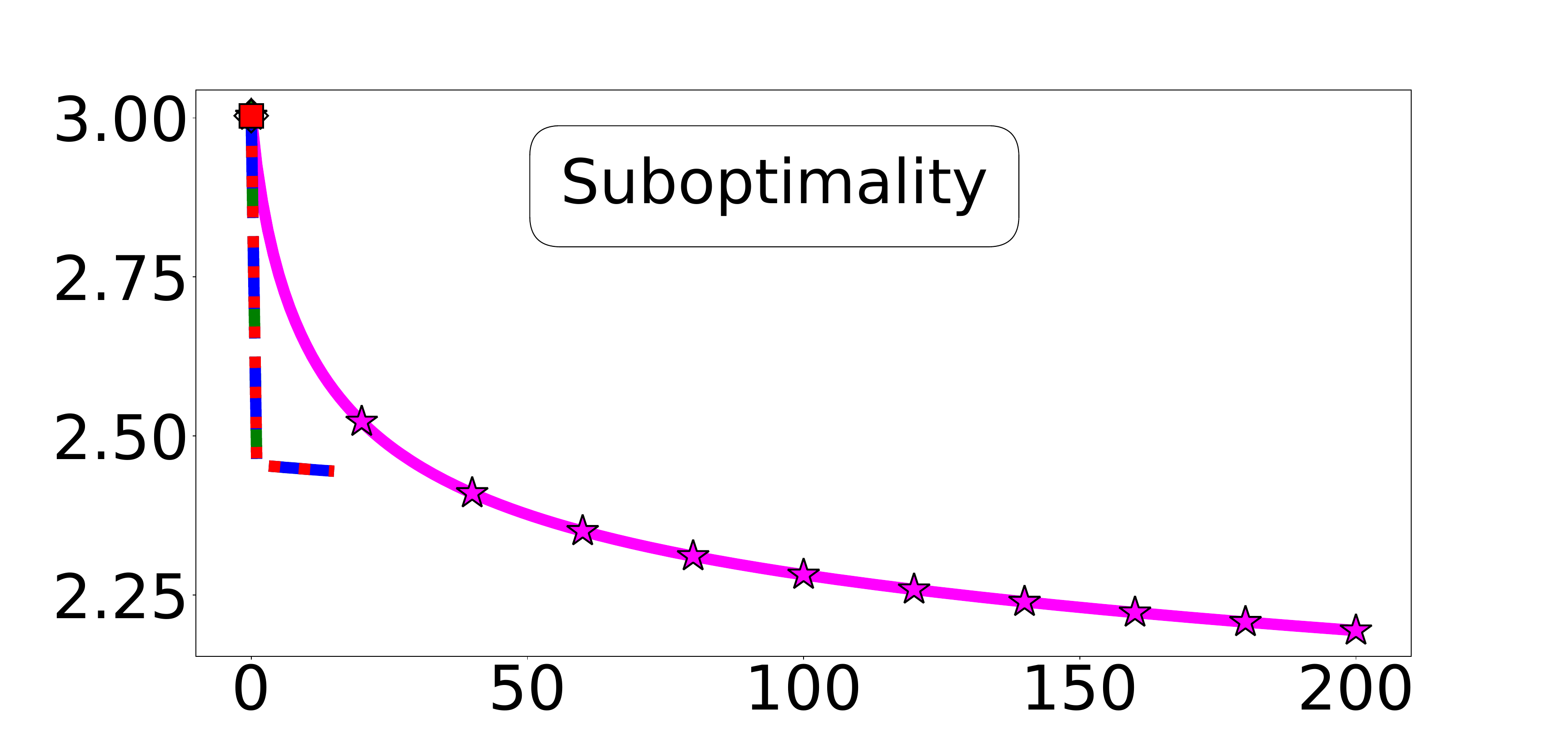}\includegraphics[width=0.5\textwidth]{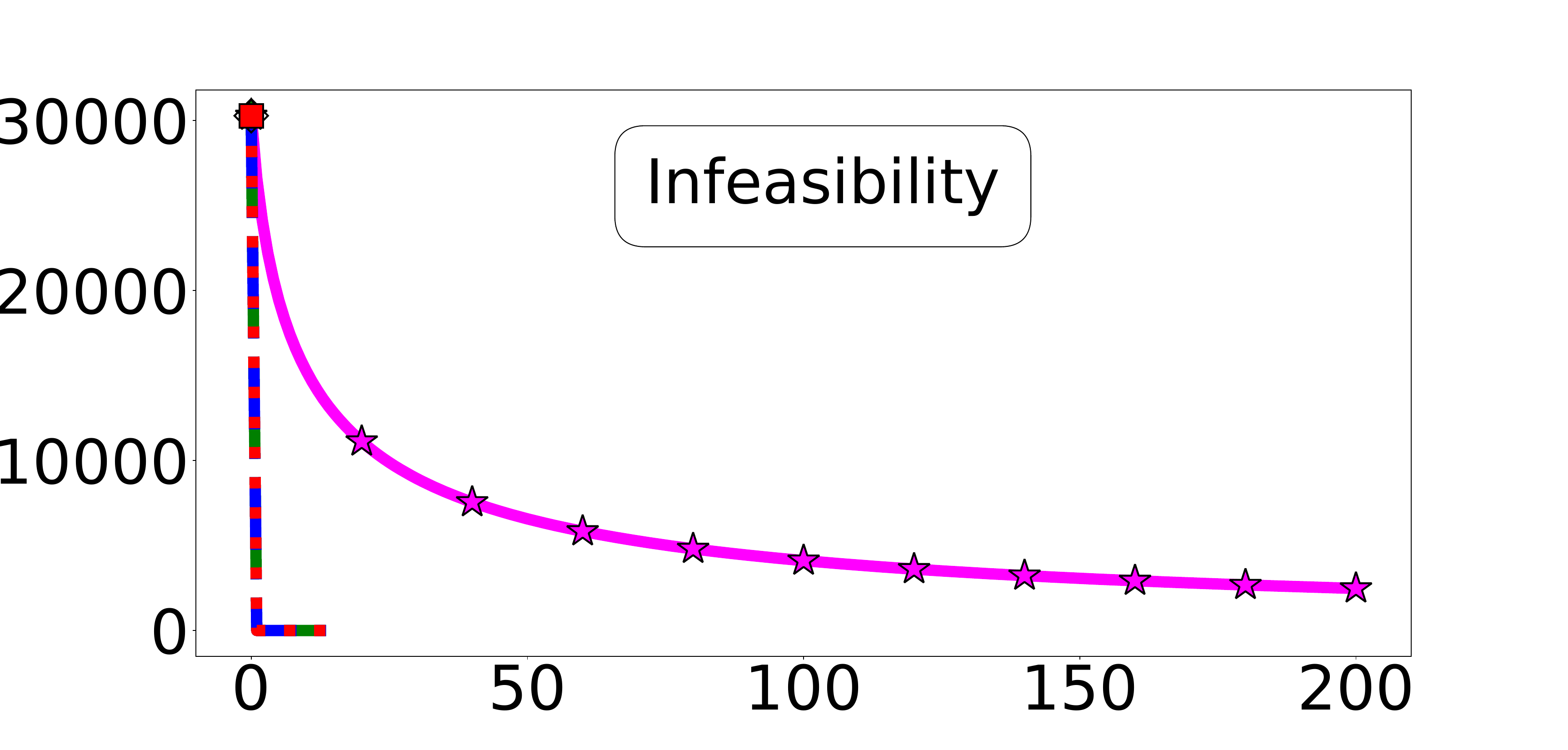}\end{minipage}
		&\begin{minipage}{7.5cm}\centering\includegraphics[width=0.5\textwidth]{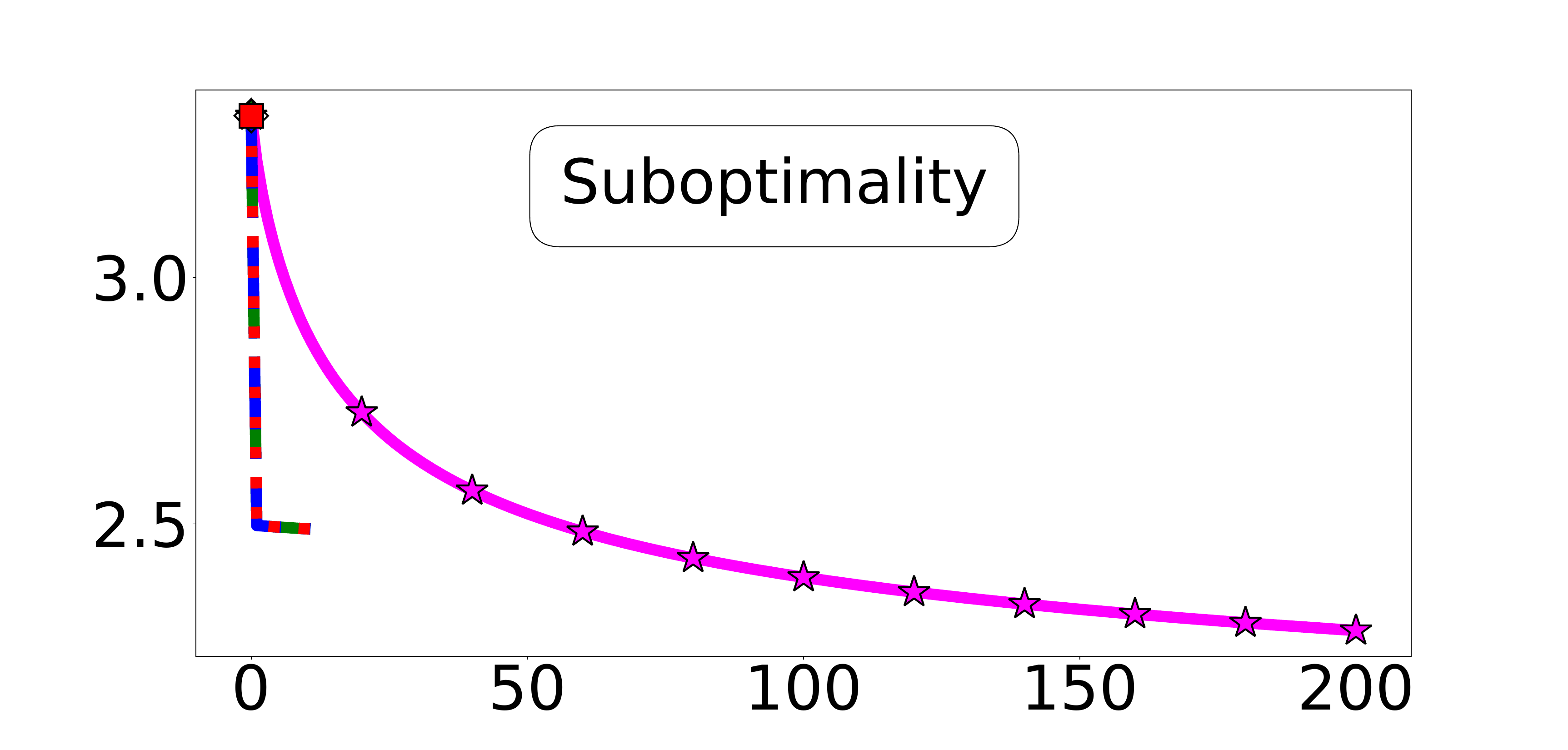}\includegraphics[width=0.5\textwidth]{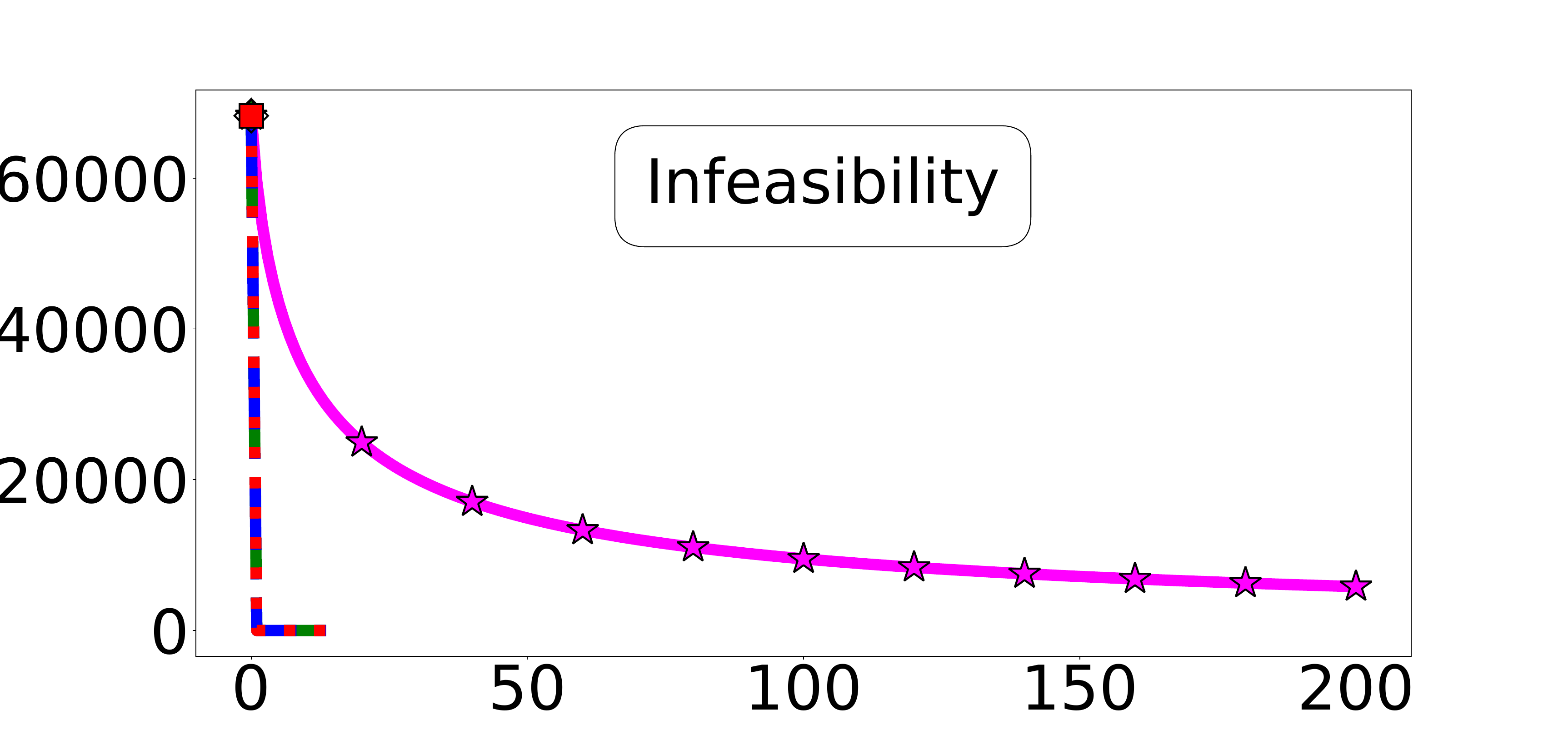}\end{minipage}\\\hline
	\end{tabular}
\captionof{figure}{Comparison of Algorithm \ref{alg:IR-IG_avg} with standard IG methods in solving an SVM model}
	\label{fig:SVM}
\end{table*}
\noindent {\bf (ii) Distributed support vector classification:} Consider a distributed soft-margin SVM problem described as follows.  Let a dataset be given as $\mathcal D \triangleq \{(u_j,v_j) \mid u_j\in \mathbb{R}^n, \ v_j \in \{-1,+1\}, \hbox{ for } j \in S\}$ where $u_j$ and $v_j$ denote the attributes and the binary label of the $j^{\text{th}}$ data point, respectively, and $S$ denotes the index set.  Let the data be distributed among $m$ agents by defining $\mathcal{D}_i $ such that $\cup_{i=1}^m \mathcal{D}_i =\mathcal D$. Let $S_i$ denote the index set corresponding to agent $i$ such that $\sum_{i=1}^{m}|S_i|=|S|$.  Then given $\lambda >0$, the distributed SVM problem is given as
	{\begin{equation}\label{prob:SVM}
			\begin{aligned}
	   & \underset{w,b,z}{\text{minimize}}&& \textstyle\sum\nolimits_{i = 1}^m\left(\mytfrac{1}{2m}\|w \|^2 + \mytfrac{1}{\lambda}\sum_{j \in S_i} z_j\right) \\
	& \text{subject to}  &&   v_j(w^T{u}_j + b)   \geq 1 - z_j,  \  \hbox{for } j \in S_i \text{ and }i \in [m], \\
	&	&& z_j \geq 0,    \qquad \qquad \qquad \ \ \ \  \text{for } j \in S_i\hbox{ and  }i \in [m].
		\end{aligned}
	\end{equation}}
To solve problem \eqref{prob:SVM} using Algorithm \ref{alg:IR-IG_avg}, we \fy{use the following result by casting \eqref{prob:SVM} as model \eqref{prob:initial_problem}.}
	\fy{\begin{lemma}\label{Lem:nonlinear_const_DO_problem} {Consider the following problem
	\begin{equation}\label{prob:subclass_constrained_opt}
		\begin{aligned}
			& {\text{minimize}} && \textstyle\sum\nolimits_{i=1}^{m}f_i(x) \\
			& \text{subject to}  &&g_{i,1}(x) \leq 0\ ,\  \ldots\ , \  g_{i,n_i}(x) \leq 0,&&  {\hbox{for } i \in [m]},\\
		&	&&  A_ix=b_i, &&  {\hbox{for }  i \in [m]}, \\
		&	&&  x \in X,
		\end{aligned}
	\end{equation}
	where agent $i$ is associated with function $f_i:\mathbb{R}^n \to\mathbb{R}$, functions $g_{i,j}:\mathbb{R}^n \to \mathbb{R}$ for $j \in {[n_i]}$, and parameters $A_i \in \mathbb{R}^{d_i\times n} $ and $b_i \in \mathbb{R}^{d_i}$. Let functions $g_{i,j}(x)$ be continuously differentiable and convex  for all {$i \in [m]$} and {$j \in[n_i]$}. Assume that the feasible region of problem \eqref{prob:subclass_constrained_opt} is nonempty and the set $X$ is closed and convex. Then, problem \eqref{prob:subclass_constrained_opt} is equivalent to \eqref{prob:initial_problem} where we define $F_i:\mathbb{R}^n \to \mathbb{R}^n$ as $F_i(x) \triangleq A_i^T(A_ix-b_i)+\textstyle\sum_{j=1}^{n_i}\max\{0,g_{i,j}(x)\}\nabla g_{i,j}(x)$.}
	\end{lemma} }
	\begin{proof}
	See Appendix \ref{app:lem:nonlinear_const_DO_problem}.
\end{proof}

We also implement some of the \fy{existing} IG schemes, namely projected IG, proximal IAG, and SAGA.  \fy{In solving \eqref{prob:SVM},} these schemes require a projection step onto the constraint set.  To \fy{implement} projections we use the Gurobi-Python solver.

{\noindent \textit{Set-up}: We consider $20$ agents and assume that $\lambda := 10$. We let $\gamma_k:=\mytfrac{1}{\sqrt{k+1}}$ and $\eta_k:=\mytfrac{1}{(k+1)^{0.25}}$ in Algorithm \ref{alg:IR-IG_avg}.  We use identical initial stepsizes in all the methods. We provide the comparisons with respect to the runtime and report the performance of each scheme over $200$ seconds.  We use a synthetic dataset with different values for $n$ and $|S|$.  The suboptimality is characterized in terms of the global objective and the infeasibility is measured by quantifying the violation of constraints of problem \eqref{prob:SVM} aligned with ideas in Lemma \ref{Lem:nonlinear_const_DO_problem}.

\noindent \textit{Insights}: In Figure \ref{fig:SVM}, we observe that with an increase in the dimension of the solution space, i.e., $n$, or the size of the training dataset, i.e., $|S|$, the projection evaluations in the standard IG schemes take longer and consequently, the performance of the IG methods is deteriorated in large scale settings. However,  utilizing the reformulation in Lemma \ref{Lem:nonlinear_const_DO_problem}, the proposed method does not require any projection operations for addressing problem \eqref{prob:SVM}.  As such, the performance of Algorithm \ref{alg:IR-IG_avg} does not get affected severely with the increase in $n$ or $|S|$.  Note that in Figure \ref{fig:SVM}, the reason that the IG schemes do not show any updates for $ |S| = 200$ and $|S|=500$ beyond a time threshold is because of the interruption in their last
update when the method reaches the $200$ seconds time limit.}

\section{Conclusion}\label{sec:concludingrem}
\fy{We consider distributed constrained optimization over the solution set of a merely monotone variational inequality problem. This formulation can capture problems with complementarity, nonlinear equality, or equilibrium constraints. We develop an iteratively regularized incremental gradient method where at each iteration agents communicate over a cycle graph.  We derive new iteration complexity bounds in terms of the global objective function and a suitably defined infeasibility metric. We validate the theoretical results on a transportation network problem and a support vector machine model. The extension to distributed methods over more general networks remains as an interesting future direction to our work. }
 
\appendix
\section{Supplementary proofs}
\subsection{Proof of Lemma \ref{Lem:nonlinear_const_DO_problem}}\label{app:lem:nonlinear_const_DO_problem}
\fy{Define $\Theta_i(x) \triangleq \mytfrac{1}{2}\|A_ix-b_i\|^2+\mytfrac{1}{2}\sum_{j=1}^{n_i} \left(\max\{0,g_{i,j}(x)\}\right)^2$ for $i\in [m]$. Note that $\Theta_i(x)$ is continuously differentiable} (see page 380 in~\cite{BertsekasNLPBook2016}). Thus, we have that $\nabla \Theta_i(x) = F_i(x)$ where $F_i(x)$ is given in the statement of Lemma \ref{Lem:nonlinear_const_DO_problem}. Next, we show that $F_i$ is a monotone mapping. From the convexity of $g_{i,j}(x)$, the function $\max\{0,g_{i,j}(x)\}$ is convex. Now, note that the function $\tfrac{1}{2}\left(\max\{0,g_{i,j}(x)\}\right)^2$ can be viewed as the composition of the nondecreasing function $h(y)\triangleq \tfrac{1}{2}y^2$ for $y \in \mathbb{R}_+$ and the convex function $\max\{0,g_{i,j}(x)\}$. Thus, $\tfrac{1}{2}\left(\max\{0,g_{i,j}(x)\}\right)^2$ is convex. This implies that $\Theta_i$ is a convex function and consequently, its gradient mapping that is $F_i(x)$, is monotone. Recalling the first-order optimality conditions for the convex optimization problems and taking into account the definition of $\text{SOL}(X,\textstyle\sum_{i=1}^m F_i)$, we have that $\text{SOL}(X,\textstyle\sum_{i=1}^m F_i) = \mathrm{argmin}_{x \in X} \textstyle\sum_{i=1}^m\Theta_i(x)$. Let $\mathcal{Z}$ denote the feasible set of problem \eqref{prob:subclass_constrained_opt}. To complete the proof, it suffices to show that $\mathcal{Z}= \mathrm{argmin}_{x \in X} \textstyle\sum_{i=1}^m\Theta_i(x)$. First, consider an arbitrary $\bar x \in \mathcal{Z}$. Then, from the definition of $\Theta_i(x)$ we must have $\Theta_i(\bar x)=0$ for all $i$, implying that $\textstyle\sum_{i=1}^m\Theta_i(\bar x) =0$. Since the feasible set of problem \eqref{prob:subclass_constrained_opt} is nonempty and that $\textstyle\sum_{i=1}^m\Theta_i(x) \geq 0$ for all $x \in X$ , we conclude that $\bar x \in \mathrm{argmin}_{x \in X} \textstyle\sum_{i=1}^m\Theta_i(x)$. Thus, we showed that $\mathcal{Z}\subseteq  \mathrm{argmin}_{x \in X} \textstyle\sum_{i=1}^m\Theta_i(x)$. Now, consider an arbitrary $\tilde x \in \mathrm{argmin}_{x \in X} \textstyle\sum_{i=1}^m\Theta_i(x)$. Thus, $\tilde x \in X$. Also, the assumption that the feasible set of problem \eqref{prob:subclass_constrained_opt} is nonempty implies that there \fy{exists} $x_0 \in X$ such that $A_ix_0=b_i$, $g_{i,j}(x_0) \leq 0$ for all {$i \in [m]$} and {$j \in [n_i]$}. Thus, we have $\textstyle\sum_{i=1}^m\Theta_i(x_0) =0$. From the nonnegativity of the function $\textstyle\sum_{i=1}^m\Theta_i(x)$ and that $\tilde x \in \mathrm{argmin}_{x \in X} \textstyle\sum_{i=1}^m\Theta_i(x)$, we must have $\textstyle\sum_{i=1}^m\Theta_i(\tilde x)=0$. Therefore, we obtain $A_i\tilde x=b_i$, $g_{i,j}(\tilde x) \leq 0$ for all {$i \in [m]$} and {$j \in [n_i]$}. Thus, we have shown that $ \mathrm{argmin}_{x \in X} \textstyle\sum_{i=1}^m\Theta_i(x)\subseteq \mathcal{Z} $. Hence, we have  $\mathcal{Z} = \text{SOL}(X,\textstyle\sum_{i=1}^m F_i)$ and the proof is completed.

\subsection{Proof of Lemma \ref{lem:stepsize_reg_parameter_sequence_a}}\label{app:lem:stepsize_reg_parameter_sequence_a}
 {\noindent (i) This result holds directly from Definition \ref{def:sequences_tikh_section} and that $\gamma_0=\mytfrac{\gamma}{\Gamma^a}$ and $\eta_0=\mytfrac{\eta}{\Gamma^b}$.

\noindent (ii) For any $k\geq 1$ we have 
\begin{align*}
1-\mytfrac{\eta_{k_2}}{\eta_{k_1}} =1-\mytfrac{\eta(k_2+\Gamma)^{-b}}{\eta(k_1+\Gamma)^{-b}} =1 - \left(\mytfrac{k_1+\Gamma}{k_2+\Gamma}\right)^b \leq 1 - \sqrt{\mytfrac{k_1+\Gamma}{k_2+\Gamma}},
\end{align*}
where the last inequality is due to $b<0.5$ and that $k_1\leq k_2$. We obtain 
\begin{align*}
1-\mytfrac{\eta_{k_2}}{\eta_{k_1}}  \leq \fy{\left(1 - {\mytfrac{k_1+\Gamma}{k_2+\Gamma}}\right)/\left(1 +\sqrt{\mytfrac{k_1+\Gamma}{k_2+\Gamma}}\right)} \leq \mytfrac{k_2-k_1}{k_2+\Gamma}.
\end{align*}

\noindent (iii) Let us use (ii) for $k_1:=k-1$ and $k_2:=k$. For all $k\geq 1$  
\begin{align*}
\mytfrac{1}{\gamma_k^3\eta_{k}} \left( 1- \mytfrac{\eta_{{k}}}{\eta_{k-1}} \right)^2& \leq \mytfrac{(k+\Gamma)^{3a}(k+\Gamma)^b}{\gamma^3\eta(k+\Gamma)^2}= \mytfrac{1}{\gamma^3\eta(k+\Gamma)^{2-3a-b}}\\
 & \leq \mytfrac{1}{\gamma^3\eta\Gamma^{2-3a-b}},
\end{align*}
where \fy{we used} $3a+b<2$, $\Gamma>0$, and that $k\geq 1$. 

\noindent (iv) For all $k\geq 1$ we can write 
\begin{align*}
\mytfrac{1}{\gamma_{k}\eta_{{k}}\mu_{\text{min}}}\left(	\mytfrac{\eta_{{k}}\gamma_{k-1}}{\gamma_{k}\eta_{{k-1}}} -1\right) 
&=\mytfrac{(k+\Gamma)^{a+b}}{\gamma\eta\mu_{\text{min}}}\left(\left(1+\mytfrac{1}{k+\Gamma-1}\right)^{a-b} -1\right)\\&\leq \mytfrac{(k+\Gamma)^{a+b}}{\gamma\eta\mu_{\text{min}}(k+\Gamma-1)},
\end{align*}
where \fy{we used} $a-b <1$, $k\geq 1$, and $\Gamma\geq 1$. We obtain
\begin{align*}
\mytfrac{1}{\gamma_{k}\eta_{{k}}\mu_{\text{min}}}\left(	\mytfrac{\eta_{{k}}\gamma_{k-1}}{\gamma_{k}\eta_{{k-1}}} -1\right) 
&\leq \mytfrac{1}{\gamma\eta\mu_{\text{min}}(k+\Gamma)^{1-a-b}}\left(1+\mytfrac{1}{\Gamma}\right)\\&\leq \mytfrac{2}{\gamma\eta\mu_{\text{min}}\Gamma^{1-a-b}} \leq 0.5,
\end{align*}
where the last two relations are implied by $\Gamma \geq 1$ and $\Gamma^{1-a-b}\geq  \mytfrac{4}{\gamma\eta\mu_{\text{min}}}$, respectively. This \fy{implies} part (iv).}



\bibliographystyle{ieeetr}
\bibliography{ref_pairIG_v01_fy}

\end{document}